\documentclass[12pt]{amsart}

\usepackage{amssymb,amscd}
\usepackage{amsbsy,amsmath,amsfonts,amsthm,extarrows,bm}
\usepackage[english]{babel}
\usepackage{graphicx}
\usepackage[all]{xy}
\usepackage[latin1]{inputenc}

\newtheorem{theorem}{Theorem}[section]
\newtheorem{lemma}[theorem]{Lemma}
\newtheorem{corollary}[theorem]{Corollary}
\newtheorem{proposition}[theorem]{Proposition}
\newtheorem{remark}[theorem]{Remark}

\newtheorem{Atheorem}{Theorem}[section]

\newtheorem{Acorollary}[Atheorem]{Corollary}

\newtheorem{Aproposition}[Atheorem]{Proposition}

\newenvironment{sproof}[1]
{\begin{proof}[#1]} {\end{proof}}

\newcommand{\ov}[1]{\overline{#1}}
\newcommand{\leftexp}[2]{{\vphantom{#2}}^{#1}{#2}}

\newcommand{\N}{\mathbb N}
\newcommand{\Z}{\mathbb Z}

\newcommand{\Q}{\mathbb Q}
\newcommand{\C}{\mathbb C}
\newcommand{\F}{\mathbb F}

\newcommand{\ovpsi}{\overline{\psi}}

\newcommand{\GL}{\textnormal{GL}}
\newcommand{\SL}{\textnormal{SL}}
\newcommand{\D}{\textnormal{D}}
\newcommand{\El}{\textnormal{E}}
\newcommand{\Ehat}{\hat{\textnormal{E}}}
\newcommand{\rk}{\textnormal{rk}}
\newcommand{\sr}{\textnormal{sr}}
\newcommand{\er}{\textnormal{er}}

\newcommand{\GE}{\textnormal{GE}}
\newcommand{\U}{\textnormal{Um}}
\newcommand{\E}{\textnormal{E}}

\newcommand{\Aut}{\textnormal{Aut}}
\newcommand{\IA}{\textnormal{IA}}
\newcommand{\Der}{\textnormal{Der}}
\newcommand{\V}{\operatorname{V}}
\newcommand{\Epi}{\operatorname{Epi}}
\newcommand{\Jac}{\mathcal{J}}

\newcommand{\nil}{\operatorname{nil}}
\newcommand{\Hom}{\operatorname{Hom}}

\newcommand{\SK}{\operatorname{SK}}
\newcommand{\dimK}{\dim_{\textnormal{Krull}}}
\newcommand{\pr}{\operatorname{pr}}
\newcommand{\T}{\operatorname{T}}

\newcommand{\Mat}{\textnormal{Mat}}

\newcommand{\br}[1]{\lbrack #1 \rbrack}
\newcommand{\Pres}[2]{\left\langle #1 \ \big\vert\  #2 \right\rangle}

\newcommand{\eb}{\mathbf{e}}
\newcommand{\gb}{\mathbf{g}}
\newcommand{\ob}{\overline{b}}
\newcommand{\ogb}{\overline{\mathbf{g}}}
\newcommand{\oP}{\overline{P}}
\newcommand{\rb}{\mathbf{r}}

\newcommand{\xb}{\mathbf{x}}

\newcommand{\mb}{\mathbf{m}}
\newcommand{\mD}{\mathcal{D}}

\newcommand{\mOD}{\mathcal{O}(\mathcal{D})}

\newcommand{\mODo}{\mathcal{O}(\mathcal{D} \setminus \{1\})}
\newcommand{\PZd}{\prod_{d \in \mD} \Z \lbrack \zeta_d \rbrack}

\newcommand{\Gab}{G_{ab}}

\newcommand{\ZX}{ \Z \lbrack  X^{\pm 1} \rbrack}
\newcommand{\ZkX}{ \Z_k \lbrack  X^{\pm 1} \rbrack}
\newcommand{\Zkl}{ \Z_k \lbrack  \Z_l \rbrack}
\newcommand{\Zwr}{ \Z_k \wr  \Z_l}
\newcommand{\Zrt}{ \Z_k \rtimes_{\alpha}  \Z_l}
\newcommand{\Zd}{ \Z \lbrack \zeta_d \rbrack}

\newcommand{\ann}{\operatorname{ann}}

\newcommand{\ZC}{ \Z \lbrack C \rbrack}
\newcommand{\Zl}{ \Z \lbrack 1/l  \rbrack}

\newcommand{\ia}{\mathfrak{a}}

\newcommand{\im}{\mathfrak{m}}
\newcommand{\tsys}{\mathfrak{t}}
\newcommand{\nielsen}{\mathfrak{n}}
\newcommand{\Ni}{\mathcal{N}}
\newcommand{\cC}{\vert C \vert}
\newcommand{\cAC}{\vert A(C) \vert}

\title{Generators of split extensions of Abelian groups by cyclic groups}
\address{EPFL ENT CBS BBP/HBP. Campus Biotech. B1 Building, Chemin des mines, 9\\Geneva 1202, Switzerland}
\email{luc.guyot@epfl.ch}
\author{Luc Guyot}
\date{\today}
\thanks{The author thanks the Mathematics Institute of G\"oettingen and the Max Planck Institute for Mathematics in Bonn for the excellent conditions provided for his stay at these institutions, during which the paper was written.}

\begin{document}
\maketitle
\begin{abstract}
Let $G \simeq M \rtimes C$ be an $n$-generator group which is a split extension of an Abelian group $M$ by a cyclic group $C$. 
We study the Nielsen equivalence classes and T-systems of generating $n$-tuples of $G$.
The subgroup $M$ can be turned into a finitely generated faithful module over a suitable quotient $R$ of the integral group ring of $C$.
When $C$ is infinite, we show that the Nielsen equivalence classes of the generating $n$-tuples of $G$ correspond bijectively 
to the orbits of unimodular rows in $M^{n -1}$ under the action of a subgroup of $\GL_{n - 1}(R)$. 
Making no assumption on the cardinality of $C$, we exhibit a complete invariant of Nielsen equivalence in the case $M \simeq R$.
As an application, we classify Nielsen equivalence classes and $\T$-systems of soluble Baumslag-Solitar groups, split metacyclic groups and lamplighter groups.
\end{abstract}

\section{Introduction} \label{SecIntro}
\subsection{Nielsen equivalence related to equivalence of unimodular rows} 
Given a finitely generated group $G$, we denote by $\rk(G)$ the minimal number of its generators. For $n \ge \rk(G)$, we let $\V_n(G)$ be the set of \emph{generating $n$-vectors} of $G$, i.e., the set of elements in $G^n$ whose components generate $G$.
In order to classify generating vectors, we can rely on a well-studied equivalence relation on $\V_n(G)$, namely the \emph{Nielsen equivalence relation}: two generating  $n$-vectors are said to be \emph{Nielsen equivalent} if they can be related by a finite sequence of  transformations taken in the set $\{ L_{ij}, I_i; 1 \le i \neq j \le n\}$ where
$L_{ij}$ and $I_i$ replace the component $g_i$ of $\gb = (g_1, \dots, g_n) \in \V_n(G)$ by $g_j g_i$ and $g_i^{-1}$ respectively and leave the other components unchanged.
We recommend \cite{Lub11, Eva07,  Pak01, LM93} to the reader interested in Nielsen equivalence and its applications.
 Let $F_n$ be the free group with basis $\xb = (x_1,\dots,x_n)$. The Nielsen equivalence relation turns out to be generated by an $\Aut(F_n)$-action. Indeed, the set
$\V_n(G)$ identifies with the set $\operatorname{Epi}(F_n,G)$ 
of epimorphisms from $F_n$ onto $G$ via the bijection $\gb \mapsto \pi_{\gb}$ with $\pi_{\gb}$ defined by $\pi_{\gb}(\xb) = \gb$.
Therefore defining $\gb \psi$ for $\psi \in \Aut(F_n)$ through 
$\pi_{\gb \psi} \Doteq \pi_{\gb} \circ \psi$ yields a right group action of $\Aut(F_n)$ on $\V_n(G)$. 
Because $\Aut(F_n)$ has a set of generators which induce the elementary Nielsen transformations $L_{ij}$ and $I_i$ \cite[Proposition 4.1]{LS77},
this action generates the Nielsen equivalence relation. 

In this article, we are concerned with finitely generated groups $G$ containing an Abelian normal subgroup $M$ and a cyclic subgroup $C$ such that $G = MC$ and $M \cap C = 1$. Denoting by $\sigma$ the natural map $G \twoheadrightarrow G/M \simeq C$, such a group $G$ fits into the split exact sequence 
\begin{equation} \label{EqExt}
0 \longrightarrow M  \longrightarrow G \xlongrightarrow{\sigma} C  \longrightarrow 1
\end{equation}
where the arrow from $M$ to $G$ is the inclusion $M \subset G$. 
The cyclic group $C = \langle a \rangle$ is finite or infinite and is given together with a generator $a$.
The action of $\Aut(F_n)$ on $\V_n(G)$ is known to be transitive if $n > \rk(G) + 2$
\cite[Theorem 4.9]{Eva93}. Our goal is to describe the $\Aut(F_n)$-orbits for the three exceptional values of $n$, namely 
$\rk(G)$, $\rk(G) + 1$ and $\rk(G) + 2$. Our main results are Theorem \ref{ThN2} and Theorem \ref{ThT2} below. They enable us to compute the exact number of Nielsen equivalence classes and $\T$-systems in a number of cases illustrated by Corollaries \ref{CorBS}, \ref{CorMetacyclic} and \ref{CorWreath}.
These theorems rely on Theorem \ref{ThMTimesC}, which relate the problem of classifying 
Nielsen equivalence classes to a pure module-theoretic problem involving $M$. The following definitions will make this relation precise. 

The conjugacy action of $C$ on $M$ defined by
$\leftexp{c}{m} \Doteq cmc^{-1}$, with $m \in M$ and $c \in C$ extends linearily to $\Z \br{C}$, turning $M$ into a module over $\Z \br{C}$. 
Let $\ann(M)$ be the annihilator of $M$. Then $M$ is a faithful module over 
$$R \Doteq \ZC/\ann(M).$$

Let $\rk_R(M)$ be the minimal number of generators of $M$ considered as an $R$-module. 
For $n \ge \rk_R(M)$, we denote by $\U_n(M)$ the set of elements in $M^n$ whose components generate $M$ 
as an $R$-module. The group $\GL_n(R)$ acts on $\U_n(M)$ by matrix right-multiplication.  There are two subgroups of $\GL_n(R)$ which are relevant to us.  The first is $\E_n(R)$, the subgroup generated by the elementary matrices, i.e., the matrices that differ from the identity by a single off-diagonal element (agreeing that $\E_1(R) = \{1\}$). The second is $\D_n(T)$, the subgroup of diagonal matrices whose diagonal coefficients belong to $T$, the \emph{group of trivial units}. We call a unit in $R^{\times} \Doteq \GL_1(R)$ a \emph{trivial unit}, if it lies in the image of $\pm C$ by the natural map $\ZC \twoheadrightarrow R$.
Theorem \ref{ThMTimesC} below establishes a connection between the $\Aut(F_n)$-orbits of generating $n$-vectors and the orbits of unimodular rows in $M$ with size $n -1$ under the action of 
$$
\Gamma_{n -1}(R) \Doteq \D_{n- 1}(T) \E_{n -1}(R).
$$

Additional definitions are needed to state this result. Denoting by $\vert C \vert$ the cardinality of $C$,
we define the \emph{norm element} of $\ZC$ to be $0$ if $C$ is infinite, and to be $1 + a + \cdots + a^{\cC - 1}$ otherwise.
Let $\nu(G)$ be the image in $R$ of the norm element of $\ZC$ via the natural map.
Let $\pi_{ab} :G \twoheadrightarrow \Gab$ be the abelianization homomorphism of $G$ and let $M_C$ be the largest quotient of $M$ with a trivial $C$-action.
We assume throughout this paper that
$n \ge \max(\rk(G), 2)$ whenever the integer $n$ refers to the size of generating vectors of $G$.
Let $\varphi_a: \U_{n - 1}(M) \rightarrow \V_n(G)$ be defined by
$\varphi_a(\mb) = (\mb, a)$. It is elementary to check that $\varphi_a$ induces a map
$$\Phi_a: \U_{n - 1}(M)/\Gamma_{n - 1}(R) \rightarrow \V_n(G)/\Aut(F_n)$$
If $\nu(G) = 0$, e.g, $C$ is infinite, then Lemma \ref{LemNu} below shows that $n > \rk_R(M)$ holds true and that $\Phi_a$ is surjective.
Our first result fully characterizes when the latter two conditions hold simultaneously.

\begin{Atheorem}[Theorems \ref{ThReductionToARow} and \ref{ThNielsenCInfinite}] \label{ThMTimesC} 
The inequality $n > \rk_R(M)$ holds and the map $\Phi_a$ is surjective if, and only if, at least one of the following holds:
\begin{itemize} 
\item[$(i)$] $n  > \rk(\Gab)$.
\item[$(ii)$] $C$ is infinite.
\item[$(iii)$] $\rk(G) > \rk(M_C) $ and $M_C$ is not isomorphic to $\Z^{\rk(G) - 1}$.
\item[$(iv)$] $\cC \in \{2, 3, 4, 6\}$ and $M_C$ is isomorphic to $\Z^{\rk(G) - 1}$.
\end{itemize}
In addition, the map $\Phi_a$ is bijective if $C$ is infinite.
\end{Atheorem}
Evidently, Theorem \ref{ThMTimesC} has no bearing on the case $n = \rk(G) = \rk_R(M)$. Proposition \ref{PropCyclicDecMaxRank} below handles this situation only when $M$ is a free module and most of our results assume that $n > \rk_R(M) $.

Combining Theorem \ref{ThMTimesC}  with various assumptions on $C$, $M$ or $R$ (e.g., $C$ is infinite and $R$ is Euclidean), we obtain a complete description of Nielsen equivalence classes of generating $n$-vectors for all $n > \rk_R(M)$. Applications to groups with arbitrary ranks are gathered in Corollaries \ref{CorGE} and \ref{CorGEMFree} below.
We present now an example of a group to which Theorem \ref{ThMTimesC} applies. 
We shall denote by $\nielsen_n(G)$ the cardinality of the set of Nielsen equivalence classes of generating $n$-vectors of $G$.
\begin{Acorollary}[Corollaries \ref{CorGE} and \ref{CorGEMFree}] \label{CorFpToDRtimesZ}
Let $p$ be a prime number and let $d \ge 1$. Let $G = \F_p^d \rtimes_A \Z$
 where $\F_p$ denotes the field with $p$ elements and where the canonical generator of $\Z$ acts on $\F_p^d$ as a matrix $A \in \GL_d(\F_p)$. Then $\nielsen_{\rk(G)}(G) = \vert R^{\times}/T \vert$ where $R = \F_p \br{X} / (P(X))$, $P(X) \in \F_p\br{X}$ is the first invariant factor of $A$ and $T$ is the subgroup of $R^{\times}$ generated by the images of $-1$ and $X$. Moreover, $\nielsen_n(G) = 1$ if $n > \rk(G)$.
\end{Acorollary}
In the above example, the polynomial $P(X)$ can be computed by means of the Smith Normal Form algorithm \cite[Section 12.2]{DF04} and, from there, an explicit formula can be derived for $\nielsen_{\rk(G)}(G)$. Indeed, if $P(X)$ is of degree $k$ and has $l$ irreducible factors with degrees $d_1, \dots, d_l$, then $\vert R^{\times} \vert =
p^{k} \prod_{i = 1}^l (1 - p^{d_i - k})$ (use for instance Lemma \ref{LemUnitsOfAFiniteRing}) while the value of $\vert T \vert$ can be deduced from the computation of the order in $\GL_k(\F_p)$ of the companion matrix of $P(X)$.

\subsection{Main results}
In this section, we make no assumption on the cardinality of $C$ but suppose that $M \simeq R$. Therefore $G \simeq R \rtimes C$ is generated by $a$ and the identity $b$ of the ring $R$. At this stage, few more examples may help understand the kind of two-generated groups we want to address. Assume that $C$ is the cyclic subgroup of $\GL_2(\Z)$ generated by an invertible matrix $a$. Let $b$ denote the $2$-by-$2$ identity matrix and let $G$ be the semi-direct product $\Z^2 \rtimes_a C$ where $a$ acts on $\Z^2$ via matrix multiplication.
It is readily checked that $\rk(G) = 2$ if and only if $M \Doteq \Z^2$ is a cyclic $\ZC$-module. If this holds, then $M$ naturally identifies with the subring $R = \Z a + \Z b$ of the ring of $2$-by-$2$ matrices over $\Z$ and we can certainly write $G \simeq R \rtimes_a C$.
If the minimal polynomial of $a$ is moreover irreducible and if $\alpha \in \C$ is one of its roots, then $G$ identifies in turn with the semi-direct product $G(\alpha) = \Z\br{\alpha^{\pm 1}} \rtimes_{\alpha} \langle \alpha \rangle$ where $\alpha$ acts on $\Z\br{\alpha^{\pm 1}} \subset \C$ via complex multiplication.
 For arbitrary choices of $\alpha \in \C$, the family $G(\alpha)$ provides us with countably many interesting non-isomorphic examples. 
For instance, if $\alpha \in \Z \setminus \{0\}$, then $G(\alpha)$ is the Baumslag-Solitar group $\Pres{a,b}{aba^{-1} = b^{\alpha}}$, which is handled in Corollary \ref{CorBS} below.
If $\alpha$ is transcendental over $\Q$, then $\Z\br{\alpha^{\pm 1}}$ is isomorphic to the ring $\ZX$ of univariate Laurent polynomials over $\Z$. In this case, the group $G(\alpha)$ is isomorphic to the restricted wreath product $\Z \wr \Z$, the subject of Corollary \ref{CorZwrZ} below.

Let us return to the presentation of our results.
Under the assumption $M \simeq R$, we shall exhibit a complete invariant of Nielsen equivalence for generating pairs. In addition, if $n = 3$ and $C$ is finite, or if $n = 4$, we prove that $\Aut(F_n)$ acts transitively on $\V_n(G)$.
Note that if $n = 3$ and $C$ is infinite, then Theorem \ref{ThMTimesC} reduces the study of the Nielsen equivalent triples to the description of the orbit set $\U_2(R)/\E_2(R)$.
Our invariant is based on a map $D$ defined as follows.
If $\nu(G) = 0$, there is a unique derivation $d \in \Der(C, R)$ satisfying $d(a)=1$ (see Section \ref{SecTsys}).  
For $\gb = (rc, r'c') \in G^2$ with $(r,r') \in R^2$ and $(c,c') \in C^2$, we set then
\begin{equation} \label{EqD1}
D(\gb) \Doteq r' d(c) - r d(c') \in R.
\end{equation} 
If $\nu(G) \neq 0$, we set furthermore
\begin{equation} \label{EqD2}
D(\gb) \Doteq D(\pi_{\nu(G) R}(\gb)) \in R/\nu(G) R,
\end{equation} where 
$\pi_{\nu(G) R}$ stands for the natural homomorphism $R \rtimes C \twoheadrightarrow (R/\nu(G) R) \rtimes C$ and
 the right-hand side of (\ref{EqD2}) is defined as in (\ref{EqD1}). The following observations will enable us to construct the desired invariant.

\begin{Aproposition}[Lemma \ref{LemmaDelta} and Proposition \ref{PropGeneratorsDelta}] \label{PropDa}
Let $G \simeq R \rtimes_{\alpha} C$ and let $\gb \in G^2$.
\begin{itemize}
\item[$(i)$] If $\gb \in \V_2(G)$, then $D(\gb) \in (R/\nu(G) R)^{\times}$.
\item[$(ii)$] Assume $\nu(G)$ is nilpotent. 
Then $\gb$ generates $G$ if and only if $\sigma(\gb)$ generates $C$ and $D(\gb) \in (R/\nu(G) R)^{\times}$.
\end{itemize}
\end{Aproposition}

Setting $$\Lambda = R / \nu(G) R, \quad T_{\Lambda} = \pi_{\nu(G) R}(T),$$ 
we define the map
$$
\begin{array}{cccc}
\Delta:&  \V_2(G) & \rightarrow & \Lambda^{\times}/ T_{\Lambda} \\
           & \gb         & \mapsto                & T_{\Lambda} D(\gb)
\end{array}
$$
The map $\Delta$ is the invariant we need and  we are now in position to describe 
the Nielsen equivalence classes of generating $n$-vectors of $G = R \rtimes C$ for $n = 2, 3$ and $4$.

\begin{Atheorem}[Theorems \ref{ThNielsenRankTwo} and \ref{ThTriplesAndQuadruples}] \label{ThN2}
Let $G = R \rtimes C$. Let $\nielsen_n(G)$ ($n =2, 3$ and $4$) be defined as in Corollary \ref{CorFpToDRtimesZ}.
Then the following hold:
\begin{itemize}
\item[$(i)$]
Two generating pairs $\gb, \gb'$ of $G$ are Nielsen equivalent if and only if
$ \pi_{ab}(\gb)$ and $\pi_{ab}(\gb')$ are Nielsen equivalent 
and 
$
\Delta(\gb) = \Delta(\gb').
$
\item[$(ii)$]
If $C$ is infinite or $\Gab$ is finite, then $\Delta$ induces a bijection  from\\
$\V_2(G)/\Aut(F_2)$ onto $\Lambda^{\times}/T_{\Lambda}$. In particular, $\nielsen_2(G) = \vert \Lambda^{\times}/T_{\Lambda} \vert$.
\item[$(iii)$]
If $C$ is finite and $\Gab$ is infinite, then $\nielsen_2(G) = \max(\frac{\varphi(\vert C \vert)}{2}, 1)\vert \Lambda^{\times}/T_{\Lambda} \vert$.
\item[$(iv)$] If $\SL_2(R) = \E_2(R)$, e.g., $C$ is finite, then $\nielsen_3(G) = 1$.
\item[$(v)$] $\nielsen_4(G) = 1$.
\end{itemize}
\end{Atheorem}

 Assertion $(i)$ of Theorem \ref{ThN2} provides us with an algorithm which decides whether or not two generating pairs of $G$ are Nielsen equivalent. Indeed, the first condition in $(i)$ can be determined by means of the Diaconis-Graham determinant \cite{DG99} while the second condition can be reduced to the ideal membership problem in $\ZX$ which is solvable \cite{PU99, Asc04}.

Consider now the left group action of $\Aut(G)$ on $\V_n(G)$ where we define $\phi \gb$ for $\phi \in \Aut(G)$ by $\pi_{\phi \gb} \Doteq \phi \circ \pi_{\gb}$, using the identification of $\V_n(G)$ with $\Epi(F_n , G)$. This action clearly commutes with the right $\Aut(F_n)$-action introduced earlier so that
$
(\phi, \psi) \gb \Doteq  \phi \gb  \psi^{-1}
$
is a left group action of  $\Aut(G)  \times \Aut(F_n)$  on $\V_n(G)$.
Following B.H. Neumann and H. Neumann \cite{NN51}, we call the orbits of this action the \emph{$\T$-systems} of generating $n$-vectors of $G$, or concisely, the \emph{$\T_n$-systems} of $G$. 
We denote by $\tsys_n(G)$ the cardinality of the $\T_n$-systems of $G$.

\begin{Atheorem} \label{ThT2}
Let $G = R \rtimes C$. 
Let $A(C)$ be set of the automorphisms of $C$ which are induced by automorphisms of $G$ preserving $M$. 
Let $A'(C)$ be the subgroup of $\Aut(C)$ generated by $A(C)$ and the involution $c \mapsto c^{-1}$.
Then the following hold. 
\begin{itemize}
\item[$(i)$] 
The cardinality $\tsys_2(G)$ is finite and we have $\tsys_2(G) \le \vert \Aut(C) / A'(C) \vert$, with equality if $R$ is a characteristic subgroup of $G$.
If $C$ is infinite or $\Gab$ is finite, then $\tsys_2(G) = 1$.
\item[$(ii)$] If $C$ is infinite and $R$ is characteristic in $G$, then we have $\vert A(C) \vert \le 2$,
$$\vert A(C) \vert \nielsen_2(G) \tsys_3(G) \ge \nielsen_3(G),\quad \vert A(C) \vert \tsys_3(G) \ge  \vert \SK_1(R)\vert,$$

 where $\nielsen_2(G)$ and $\nielsen_3(G)$ are as in Corollary \ref{CorFpToDRtimesZ} and $\SK_1(R)$ denotes the special Whitehead group of $R$ (see Section \ref{SecN2N3} for a definition of $\SK_1$).
\end{itemize}
\end{Atheorem}

Note that assertion $(i)$ of Theorem \ref{ThT2} generalizes Brunner's theorem \cite[Theorem 2.4]{Bru74} according to which a two-generated Abelian-by-(infinite cyclic) group has only one $\T_2$-system. We give a wider generalization of Brunner's theorem with Theorem \ref{ThOneTSystem} below.

\subsection{Applications}
Our first corollary shows that there is no upper bound for $\tsys_2(G)$ when $G$ ranges in the class of two-generated split Abelian-by-cyclic groups.
\begin{Acorollary}[Corollary \ref{CorTTwoSystemsExample}] \label{ThT2GN}
For every integer $N \ge 1$, there exists a group $G_N$ of the form $R \rtimes C$ with $C$ finite such that $\tsys_2(G_N) \ge N$,
where $\tsys_2$ is defined as in Theorem \ref{ThT2}.
\end{Acorollary}

For comparison, Dunwoody constructed for every $N \ge 1$ a
two-step nilpotent $2$-group $D_N$ on two generators such that $\tsys_2(D_N) \ge N$ \cite{Dun63}.


Theorems \ref{ThN2} and \ref{ThT2} can be used to compute the number of generating pairs of $G = R \rtimes C$ in each of its Nielsen equivalence classes when $G$ is finite.
\begin{Acorollary} \label{CorV2}
Let $G = R \rtimes C$ and assume that $G$ is finite. Then every Nielsen class of generating pairs has the same number of elements and
$$\vert \V_2(G) \vert = \frac{\vert R \vert  \vert R^{\times} \vert}{ \vert R_C \vert  \vert (R_C)^{\times} \vert}
\vert \V_2(\Gab) \vert$$ where $R_C$ is the largest quotient ring of $R$ with trivial $C$-action.
\end{Acorollary}
In the above identity, the cardinality $\left\vert \V_2(\Gab) \right\vert$ can be computed using the formulas of \cite[Remark 1]{DG99}. In many instances, all terms in Corollary \ref{CorV2}'s formula can be computed. For example, let $G = \F_q \rtimes \F_q^{\times} $ where $\F_q$ is the field with $q = p^k$ elements for $p$ a prime number. Then Corollary  \ref{CorV2} yields $\vert \V_2(G) \vert = \frac{q(q -1)}{p(p -1)} \vert \V_2(\Z/p(q -1)\Z) \vert$. 

We now turn to applications of Theorems \ref{ThN2} and \ref{ThT2} for three different classes of two-generated groups, namely the soluble Baumslag-Solitar groups, the split metacyclic groups and the lamplighter groups.
A \emph{Baumslag-Solitar group} is a group with a presentation of the form 
$$BS(k, l) =\Pres{a,b}{ab^ka^{-1} = b^l}$$ for $k, l \in \Z \setminus \{0\}$.
Brunner proved that $BS(2,3)$ has infinitely many $\T_2$-systems whereas its largest metabelian quotient, namely $G(2/3) = \Z\br{1/6} \rtimes_{2/3} \Z$, has only one $\T_2$-system \cite[Theorem 3.2]{Bru74} . 
The group $BS(k, l)$ is soluble if and only if $\vert k \vert = 1$ or $\vert l \vert = 1$. As a result, a soluble Baumslag-Solitar group is isomorphic to $BS(1, l)$ for some $l \in \Z \setminus \{0\}$ and hence admits a semi-direct decomposition 
$\Z\br{1/l} \rtimes_l \Z$ where the canonical generator of $\Z$ acts as the multiplication by $l$ on 
$\Z\br{1/l} = \{\frac{z}{l^i}; \,z \in \Z, i \in \N\}$.

\begin{Acorollary}\label{CorBS}
Let $G = BS(1, l)$ with $l \in \Z \setminus \{0\}$ and let $\nielsen_n, \tsys_n$ ($n = 2,3$) be defined as in Corollary \ref{CorFpToDRtimesZ} and Theorem \ref{ThT2}. Then the following hold:
\begin{itemize}
\item[$(i)$]
$\nielsen_2(G)$ is finite if and only if $l = \pm p^d$ for some prime number $p \in \N$ and some non-negative integer $d$. 
In this case, $\nielsen_2(G) = \max(d, 1)$.
\item[$(ii)$]
$\tsys_2(G) = \nielsen_3(G) = 1$.
\end{itemize}
\end{Acorollary}

Define $\Z_d = \Z/d\Z$ for $d \ge 0$ (thus $\Z_0 = \Z$) and let $\varphi(d)$ be the cardinality of $\Z_d^{\times}$.
A \emph{split metacyclic group} is a semi-direct product of the form
$\Zrt$ with $k, l \ge 0$. Here the canonical generator of $\Z_l$ is denoted by $a$ and acts on $\Z_k$ as the multiplication by $\alpha \in \Z_k^{\times}$.
\begin{Acorollary} \label{CorMetacyclic}
Let $G = \Zrt$ ($k, l \ge  0$) and  let $\nielsen_n, \tsys_n$ ($n = 2,3$) be defined as in Corollary \ref{CorFpToDRtimesZ} and Theorem \ref{ThT2}.
Then the following hold.
\begin{itemize}
\item[$(i)$] If $k = 0$ and $\alpha = 1$, then $\nielsen_2(G) = \max(\varphi(l)/2, 1)$.
\item[$(ii)$] Assume that $k \neq 0$ or $\alpha \neq 1$. Then 
$\nielsen_2(G) = \frac{\varphi(\lambda)}{\omega}$, where $\lambda \ge 0$ is such that $\Z_{\lambda} \simeq \Z_k/\nu(G) \Z_k$ and $\omega$ is the order of the subgroup of $\Z_{\lambda}^{\times}$ generated by $-1$ and the image of $\alpha$.
\item[$(iii)$]
$\tsys_2(G) = \nielsen_3(G) = 1$.
\item[$(iv)$]
If $G$ is finite, then $\left\vert \V_2(G) \right\vert = \frac{k\varphi(k)}{e\varphi(e)} \left\vert \V_2(\Z_e \times \Z_l) \right\vert$, where 
$e \ge 1$ is such that $\Z_e \simeq \Z_k/(1 - \alpha) \Z_k$. In addition, every Nielsen equivalence class of generating pairs has 
the same number of elements.
\end{itemize}
\end{Acorollary}
Our Corollary \ref{CorMetacyclic} applies for instance to dihedral groups and to almost all $p$-groups with a cyclic subgroup of index $p$ \cite[Theorem IV.4.1]{Bro82}.

A two-generated \emph{lamplighter group} is a restricted wreath product of the form
$\Z_k \wr \Z_l$ with $k, l \ge 0$.
Such a group reads also as 
$R \rtimes_a C$ with
$C = \Z_l$ and 
$R = \Z_k \br{C}$, the integral group ring of $C$ over $\Z_k$.
We are able to determine the number of Nielsen classes and $\T$-systems of any two-generated 
lamplighter groups with the exception of $\Z \wr \Z$. 
\begin{Acorollary}[Corollaries \ref{CorWreathOneTTwoSystem} and \ref{CorWreathNTwo}] \label{CorWreath}
Let $G = \Z_k \wr \Z_l$ ($k, l \ge 0$ and $k, l \neq 1$) and  let $\nielsen_n$ and $\tsys_n$ ($n = 2,3$) be defined as in Corollary \ref{CorFpToDRtimesZ} and Theorem \ref{ThT2}. Then the following hold.
\begin{itemize}
\item[$(i)$]
$\tsys_2(G) = 1$.
\item[$(ii)$]
If $\Z_k$ or $\Z_l$ is finite, then $\nielsen_3(G) = 1$.
\item[$(iii)$]
Assume that $\Z_k$ is finite and $\Z_l$ is infinite. Then 
$\nielsen_2(G)$ is finite if and only if $k$ is prime; in this case $\nielsen_2(G) = \max(\frac{k - 1}{2}, 1)$.
\item[$(iv)$]
Assume that $\Z_k$ is infinite and $\Z_l$ is finite. Then
$\nielsen_2(G)$ is finite if and only if $l \in \{2, 3, 4, 6\}$; in this case $\nielsen_2(G) = 1$.
\end{itemize}
\end{Acorollary}
The case of finite two-generated 
lamplighter groups is addressed by Corollary \ref{CorFiniteWreathProduct} below and a formula for $\left\vert \V_2(\Z_k \wr \Z_l)\right\vert$ is derived. 
For the restricted wreath product $\Z \wr \Z$, we show that the problem of classifying Nielsen equivalence classes and $\T$-systems of generating triples is tightly related to an open problem in ring theory.
\begin{Acorollary} \label{CorZwrZ}
Let $G = \Z \wr \Z$ and let $\nielsen_2, \nielsen_3$ and $\tsys_3$ be defined as in Corollary \ref{CorFpToDRtimesZ} and Theorem \ref{ThT2}. Then we have $\nielsen_2(G) = 1$ and $\nielsen_3(G) \le 2 \tsys_3(G)$. In addition, the following are equivalent:
\begin{itemize}
\item[$(i)$] $\nielsen_3(G) = 1$.
\item[$(ii)$] $\tsys_3(G) = 1$.
\item[$(iii)$] The ring $R$ of univariate Laurent polynomials over $\Z$ satisfies $\SL_2(R) = \E_2(R)$ 
(cf. \cite[Conjecture 5.3]{Abr08}, \cite[Open Problem MA1]{BMS02}, \cite[Open problem]{BM82}).
\end{itemize}
\end{Acorollary}

The paper is organized as follows. Section \ref{SecPreliminary} deals with notation and gathers known facts on rings which are ubiquitous in our presentation: the generalized Euclidean rings and the quotients of the ring of univariate Laurent polynomials over $\Z$. Section \ref{SecMC} is dedicated to the proof of Theorem \ref{ThMTimesC}. In Section \ref{SecARow}, we determine the conditions under which the map $\Phi_a$ of Theorem \ref{ThMTimesC} is surjective while Section \ref{SecCIsZ} addresses the case $C \simeq \Z$ for which it is shown that $\Phi_a$ is bijective. 
Section \ref{SecRC} is dedicated to the proofs of Theorems \ref{ThN2}, \ref{ThT2}
 and \ref{ThT2GN}.  Section \ref{SecLamplighterGroups} presents the proof of Corollary \ref{CorV2} and applications to Baumslag-Solitar groups, split metacyclic groups and lamplighter groups, i.e., Corollaries \ref{CorBS}, \ref{CorMetacyclic} and \ref{CorWreath}.

\paragraph{\textbf{Acknowledgments}.}
The author is grateful to Pierre de la Harpe, Tatiana Smirnova-Nagnibeda and Laurent Bartholdi for encouragements, 
useful references and comments made on preliminary versions of this paper.

\section{Preliminary results} \label{SecPreliminary}

\subsection{Notation and definitions} \label{SecNotation}
We set in this section the notation and the definitions used throughout the article. 
A parallel is drawn between generating vectors of a group and unimodular rows of a module.
\paragraph{\textbf{Rings}}
All considered rings are commutative rings with identity.
Given a ring $R$, we denote by  $\Jac(R)$ its \emph{Jacobson radical}, i.e., the intersection of all its maximal ideals. We denote by  $\nil(R)$ the \emph{nilradical} of $R$, i.e., the intersection of all its prime ideals.
 The nilradical coincides with the set of nilpotent elements \cite[Theorem 1.2]{Mat89}. 
Let $M$ be a finitely generated $R$-module. Then an $R$-epimorphism of $M$ is an $R$-automorphism \cite[Theorem 2.4]{Mat89}, a fact that we will use without further notice. Let $N \subset M$ be finitely generated $R$-modules and $I \subset \Jac(R)$ an ideal of $R$. Then the identity $N + IM = M$ implies $N = M$. We refer to the latter fact as Nakayama's lemma \cite[Theorem 2.2's corollary]{Mat89}.
Apart from Section \ref{SecGE}, the ring $R$ will always be a quotient of $\ZX$, the ring of univariate Laurent polynomials over $\Z$. In this case, we have $\nil(R) = \Jac(R)$ \cite[Theorem 4.19]{Eis95} and we shall consistently denote by $\alpha$ the image of $X$ under the quotient map. We set $\Z_d \Doteq \Z/d\Z$  for $d \ge 0$. Thus the additive group of $\Z_d$ is the cyclic group with $d$ elements if $d > 0$ whereas $\Z_0 = \Z$. We denote by $\varphi$ the Euler totient function, so that $\varphi(d) = \vert \Z_d^{\times}\vert$ if $d > 0$. We set furthermore $\varphi(0) \Doteq 2$.

\paragraph{\textbf{Orbits of generating vectors}}
Let $G$, $H$ be groups and let $f \in \Hom(G, H)$. We denote by $1_G$ the trivial element of $G$. For $\gb =(g_i) \in G^n$, we set 
$f(\gb) \Doteq (f(g_i))$.
Thus the component-wise left action of $\Aut(G)$ on  $\V_n(G)$ reads as 
$\phi \gb = \phi(\gb)$ for $(\phi ,\gb) \in \Aut(G) \times\V_n(G)$. 
This action clearly coincides with the $\Aut(G)$-action introduced earlier.
Let us examine the component-wise counterpart of the $\Aut(F_n)$-action we previously defined 
via the identification of $\V_n(G)$ with $\Epi(F_n, G)$. 
For $\psi \in \Aut(F_n)$, set $w_i(\xb) \Doteq \psi(x_i) \in F_n$ for $i=1,\dots,n$.
Then the $\Aut(F_n)$-action reads as $ \gb  \psi = (w_i(\gb))$.

\paragraph{\textbf{Orbits of unimodular rows}}
 For $r \in R$ and $i \neq j$, we denote by $E_{ij}(r) \in \GL_n(R)$ the elementary matrix with ones on the diagonal and whose $(i,j)$-entry is $r$, all other entries being $0$. 
For $u \in R^{\times}$, we denote by $D_i(u) \in \GL_n(R)$ the diagonal matrix with ones on the diagonal except at the $(i,i)$-entry, which is set to $u$. Recall that $\El_n(R)$ is the subgroup of $\GL_n(R)$ generated by elementary matrices. Given a subgroup $U$ of $R^{\times}$, we define $\D_n(U)$ as the subgroup of $\GL_n(R)$ generated by the matrices $D_i(u)$ with $u \in U$. Following P.M. Cohn \cite{Coh66}, we denote by $\GE_n(R)$ the subgroup generated by $\El_n(R)$ and $\D_n(R^{\times})$.
Let $\Gamma$ be a subgroup of $\GL_n(R)$ and let $M$ be an $R$-module. 
Two rows $\rb,\rb' \in M^n$ are said to be \emph{$\Gamma$-equivalent}  
if there is $\gamma \in \Gamma$ such that $\rb' = \rb \gamma$. 
A row is termed \emph{unimodular} if its components generate $M$
as an $R$-module. By definition, $\U_n(M)$ is the set of unimodular rows of $M$.

\paragraph{\textbf{Elementary rank and stable rank}} 
We say that $n \ge 1$ belongs to the \emph{elementary range} of $R$ if $\E_{n + 1}(R)$ acts transitively on $\U_{n + 1}(R)$. 
The \emph{elementary rank} of $R$ is the least integer $\er(R)$ such that $n$ lies in the elementary range of $R$ for every 
$n \ge \er(R)$. 
We say that $k$ lies in the \emph{stable range} of $R$ if for every 
$n \ge k$ and every $(r_i) \in \U_{n + 1}(R)$ there is $(s_i) \in R^n$ such that 
$(r_1 + s_1 r_{n + 1},r_2 + s_2 r_{n + 1}, \dots, r_n + s_n r_{n + 1}) \in \U_n(M)$. 
The \emph{stable rank} of $R$ is the least integer $\sr(M)$ lying in the stable range of $M$.
By \cite[Proposition 11.3.11]{McCR87} we have:
\begin{equation} \label{EqStableRank}
1 \le \er(R) \le \sr(R).
\end{equation}
If $R$ is moreover Noetherian, the Bass Cancellation Theorem asserts \cite[Corollary 6.7.4]{McCR87} that
\begin{equation} \label{EqBassCancellationTh}
\sr(R) \le  \dimK (R) + 1.
\end{equation}

\subsection{$\GE$-rings} \label{SecGE}
Our study of Nielsen equivalence is significantly simplified when dealing with rings $R$ which are similar to Euclidean rings in a specific sense.
 A ring $R$ is termed a \emph{$\GE_n$-ring} if $\GE_n(R) = \GL_n(R)$, which is equivalent to saying that $\SL_n(R) = \El_n(R)$. 
Indeed, we have
$
\GE_n(R) = \D_n(R^{\times}) \El_n(R)
$
 and a matrix $D \in \D_n(R^{\times})$ lies in $\SL_n(R)$ if and only if 
 it lies in $\El_n(R)$ by Whitehead's lemma \cite[Lemma 1.3.3]{Wei13}. Thus the latter equality implies the former, the converse being obvious.
A ring $R$ is called a \emph{generalized Euclidean ring} in the sense of P. M. Cohn \cite{Coh66}, 
or a $\GE$-ring for brevity, if it is a $\GE_n$-ring  for every $n \ge 1$. Euclidean rings are known to satisfy this property \cite[Theorem 4.3.9]{HO89}.
The reader is invited to check the two following elementary lemmas.
\begin{lemma} \label{LemGECriteria}
The following assertions hold:
\begin{itemize}
\item[$(i)$] $R$ is a $\GE_2$-ring if and only if $1$ lies in the elementary range of $R$, i.e, $\E_2(R)$ acts transitively on 
$\U_2(R)$. 
\item[$(ii)$] If $\er(R) = 1$ then $R$ is a $\GE$-ring. In particular, $R$ is a $\GE$-ring if $\sr(R) = 1$.
\end{itemize}
\end{lemma}

A semilocal ring, i.e., a ring with only finitely many maximal ideals, has stable rank $1$ \cite[Corollary 6.5]{Bas64}. 
As a result, semilocal rings, and Artinian rings in particular, are $\GE$-rings.

\begin{lemma} \label{LemStabilityOfGE}
The following assertions hold:
\begin{itemize}
\item[$(i)$]  Let $J$ be an ideal contained in $\Jac(R)$. Then $R$ is a $\GE$-ring if and only if $R/J$ is a $\GE$-ring
\cite[Proposition 5]{Gel77}.
\item[$(ii)$] Assume $R$ is a direct product $\prod_{i = 1}^N R_i$. 
Then $R$ is a $\GE$-ring if and only if each factor ring $R_i$ is a $\GE$-ring. \cite[Theorem 3.1]{Coh66}
\end{itemize}
\end{lemma}

\begin{lemma} \label{LemArtinianCoefficientAndGE}
Let $R$ be an Artinian ring. 
Then every homomorphic image of
$R\lbrack X^{\pm1} \rbrack$ is a $\GE$-ring.
\end{lemma}

\begin{sproof}{Proof of Lemma \ref{LemArtinianCoefficientAndGE}}
Since $\Jac(R) = \nil(R)$, we have $\Jac(R) \lbrack X^{\pm1} \rbrack \subset  \Jac(R\lbrack X^{\pm1} \rbrack)$. 
As the factor ring 
$P \Doteq R\lbrack X^{\pm1} \rbrack / \Jac(R) \lbrack X^{\pm1} \rbrack$ 
is isomorphic to a direct product of finitely many Euclidean rings, we deduce from Lemma \ref{LemArtinianCoefficientAndGE} that 
$R\lbrack X^{\pm1} \rbrack$ is a $\GE$-ring. Let us consider a quotient $Q$ of $R\lbrack X^{\pm1} \rbrack$. Then $Q/\Jac(Q)$ is a quotient of $P$ and is therefore a direct product of finitely many Euclidean and Artinian rings. As a result, the ring $Q$ is a $\GE$-ring. 
\end{sproof}

\begin{remark} \label{RemStableRankReduction}
If $\sr(R) = r < \infty$, 
then it easy to prove that any matrix in $\GL_n(R)$ 
for $n > r$ can be reduced to a matrix of the form 
$
\begin{pmatrix}
A & 0 \\
0 & I_{n - r}
\end{pmatrix}
$
with $A \in \GL_r(R)$ by elementary row transformations. Thus $R$ is a $\GE$-ring if it is a $\GE_n$-ring for every $n \le r$.
\end{remark}

\subsection{The ring of univariate Laurent polynomials and its quotients} \label{SecZC}
Because the structure of the $R$-module $M$ fully determines $G$, and because $R$ is a quotient of $\Z\br{\Z} \simeq \ZX$, the ring of univariate Laurent polynomials over $\Z$ plays a prominent role in this article.
In this section, we collect preliminary facts about $\ZX$ and its quotients. These facts pertain to row reduction of matrices over $R$ and unit group description; they enable us to count precisely Nielsen equivalence classes in our applications.

A ring $R$ is said to be \emph{completable} if every unimodular row over $R$ can be completed into an invertible square matrix over $R$,
or equivalently, if $\GL_n(R)$ acts transitively on $\U_n(R)$ for every $n \ge 1$.

\begin{lemma} \label{LemCompletable}
Every homomorphic image of $\ZX$ is completable.
\end{lemma}

\begin{proof}
If $R$ is isomorphic to $\ZX$, then $R$ is completable by \cite[Theorem 7.2]{Sus77}. 
So we can assume that $\dimK(R) \le 1$. 
Let $n \ge 2$ and $\rb \in \U_n(R)$. Since $\sr(R) \le 2$ by (\ref{EqBassCancellationTh}), we can find $E \in \E_n(R)$ 
such that $\rb E = (r_1, r_2, 0, \dots, 0)$. 
Let $A \in \SL_2(R)$ be such that $(r_1, r_2)A = (1, 0)$ and set 
 $B = \begin{pmatrix}
 A & 0 \\
 0 & I_{n -2}
 \end{pmatrix}
 $.
 Then we have $\rb EB = (1,0,\dots,0)$.
\end{proof}

The following shows that Theorem \ref{ThN2}.$iv$ applies whenever $C$ is a finite cyclic group.
\begin{theorem} \label{ThZCGE} \cite[Theorem A]{Guy16c}
Let $C$ be a finite cyclic group. Then every homomorphic image of $\Z \br{C}$ is a $\GE$-ring.
\end{theorem}

The next lemma will come in handy when scrutinizing lamplighter groups in Section \ref{SecLamplighterGroups}. Before we can state this lemma, we need to introduce some notation.
Given a rational integer $d > 0$, we set $\zeta_d \Doteq e^{\frac{2i\pi}{d}}$ and let 
$$\lambda_d: \Z\br{X} \rightarrow \Zd$$ be the ring homomorphism induced by the map $X \mapsto \zeta_d$.
Given a set $\mD$ of positive rational integers, we define 
$$\lambda_{\mD} : \Z\br{X} \rightarrow \PZd$$ by 
$\lambda_{\mD} \Doteq  \prod_{d \in \mD} \lambda_d$ and set $\mOD \Doteq \lambda_{\mD}(\ZX)$.
Let $\alpha \Doteq \lambda_{\mD}(X) \in \mOD$.
Recall that a unit in $\mOD^{\times}$ is said to be trivial if it lies in $T$, the subgroup generated by $-1$ and $\alpha$.
\begin{lemma} \label{LemUnitsOfOD}
Let $\mD$ be a non-empty finite set of positive rational integers.
\begin{enumerate}
\item The torsion-free rank of $\mOD^{\times}$ is $\sum_{d \in \mD,\, d > 2} (\frac{\varphi(d)}{2} - 1)$.
\item
Assume $\mD$ is the set of divisors of $l$, with $l \ge 2$.
\begin{itemize}
\item[$(i)$]
Any non-trivial unit of finite order in $\mODo$ is of the form $u\left(1 + \sum_{i \in E} \alpha^i\right)$
for some trivial unit $u$ and some non-empty subset $E \subset \{ 1, 2, \dots, l - 1\}$.
\item[$(ii)$]
If $l \in \{2, 3, 4, 6\}$, then the units of $\mODo$ are trivial.   
\end{itemize}
\end{enumerate}
\end{lemma}

\begin{proof}
The proofs of $1$ and $2$.i essentially adapt \cite[Theorems 3 and 4]{AA69} 
to the rings $\mOD$ under consideration; we provide them for the reader's convenience.

$1$.
Since the additive groups of $R = \PZd$ and $\mOD$ are free Abelian groups of the same rank $r = \sum_{d \in \mOD} \varphi(d)$, 
the latter group is of finite index $k$ in the former for some $k \ge 1$. By Dirichlet's Unit Theorem, the group $R^{\times}$ is finitely generated and its 
torsion-free rank is $\sum_{d \in \mD,\, d > 2} (\frac{\varphi(d)}{2} - 1)$. Therefore, it is sufficient to prove that 
$\mOD^{\times}$ is of finite index in $R^{\times}$. This certainly holds if every unit $u \in \Zd$ for $d \in \mD$ is of finite order modulo 
$\mOD$. To see this, consider the principal ideal $I$ of $\Zd$ generated by $k$. Since $\Zd / I$ is finite, there is $k' \ge 1$  such that $u^{k'} \equiv 1 \mod I$. Therefore $u^{k'} = 1 + k \zeta$ for some $\zeta \in \Zd$. As $k \zeta \in \mOD$, we deduce that $u^{k'} \in \mOD$, which completes the proof.

$2.i$. 
Let $g \in \mOD$ such that the projection $\pr_1: \mOD \twoheadrightarrow \mODo$ maps $g$ to a unit of finite order. 
 Identifying $\mOD$ with $\ZC$ for $C = \Z/l\Z$, we can write $g = \sum_{c \in C} a_c c$ with $a_c \in \Z$.
For $d$ dividing $l$, let $\pi_d$ be the projection of $\mOD$ onto $\Zd$, let $\chi_d = {\pi_d}_{\vert C}$ and 
$\rho_d = \pi_d(g) = \sum_{c \in C} a_c \chi_d(c)$. Since $\pr_1(g)$ is of finite order, $\rho_d$ is a root of unity for every 
divisor $d > 1$. The characters $\chi_d$ form a complete set of inequivalent characters of $C$ by \cite[Lemma 2]{AA69}.
Therefore, we have 
$$
\sum_{c \in C} a_c \chi(c) = \rho_{\chi}
$$
for every $\chi \in \hat{C}$, the character group of $C$, where $\rho_{\chi}$ is a root of unity if $\chi \neq 1$.
Using the orthogonality relation of characters, we obtain
$$
l a_c = \sum_{\chi \in \hat{C}} \rho_{\chi} \ov{\chi(c)}
$$
for every $c \in C$. Hence $\vert a_c - a_{c'} \vert \le \frac{2(l - 1)}{l} < 2$ for every $(c, c') \in C^2$.
Replacing $g$ by $\varepsilon (g - k \sum_{c \in C}c)$ for a suitable choice of $\varepsilon \in \{\pm 1\}$ and  $k \in \Z$,
 we can assume that $a_c \in \{0, 1\}$ for every $c$. Replacing $g$ by $cg$ for some suitable choice of $c \in C$, we can assume that
$a_{1_C} = 1$. This ensures eventually that the image of $g$ in $\mODo$ has the desired form.

$2.ii$. If $l \in \{2, 3\}$, then $\mODo = \Z\br{\zeta_l}$ and this ring has only trivial units. Assume now $l = 4$. Since 
$R = \mathcal{O}(\{2, 4\})$ embeds into 
$\Z \times \Z\br{i}$, it has at most $8$ units. It is easily checked that there are exactly $8$ trivial units in $R$. 
Therefore all units are trivial. Assume eventually that $l = 6$. Since 
$R = \mathcal{O}(\{2, 3, 6\})$ embeds into 
$\Z \times \Z\br{\zeta_3} \times \Z\br{\zeta_3}$, it has only units of finite orders. 
Considering projections on each of the three factors, it is routine to check that no element of the form
$1 + \sum_{i \in E} \alpha^i$ with $\emptyset \neq E \subset \{1, 2, 3, 4 ,5\}$ is a unit in $R$.
This proves that $R$ has only trivial units by 2.$i$.
\end{proof}

\subsection{Nielsen equivalence in finitely generated Abelian groups} \label{SecNielsenAbel}
In this section, we present the classification of generating tuples modulo Nielsen equivalence in finitely generated Abelian groups. This result is instrumental in Section \ref{SecARow} when reducing generating vectors to a standard form.
Different parts of the aforementioned classification were obtained by different authors: \cite[Satz 7.5]{NN51}, \cite[Section 2's lemmas]{Dun63}, \cite[Lemma 4.2]{Eva93}, \cite[Example 1.6]{LM93} and \cite{DG99}. The classification reaches its complete form with

\begin{theorem}  \cite[Theorem 1.1]{Oan11} \label{ThNielsenAbel}
Let $G$ be a finitely generated Abelian group whose invariant factor decomposition is
$$
\Z_{d_1} \times \cdots  \times \Z_{d_k},  \quad 1 \neq d_1 \, \vert \, d_2 \, \vert \, \cdots \, \vert \, d_k, \, d_i \ge 0
$$

Then every generating $n$-vector $\gb$ with $n \ge k = \rk(G)$ is
Nielsen equivalent to
$(\delta e_1, e_2, \dots, e_k, 0, \dots, 0)$ for some $\delta \in \Z_{d_1}^{\times}$ 
 where $\eb = (e_i) \in G^k$ is defined by $(e_i)_i = 1 \in \Z_{d_i}$ and $(e_i)_j = 0$ for $j \neq i$. 
Furthermore, the following hold:
\begin{itemize}
\item If $n > k$, then we can take $\delta = 1$.
\item If $n = k$, then $\delta$ is unique, up to multiplication by $-1$.
\end{itemize}

In particular, $G$ has only one Nielsen equivalence class of generating $n$-vectors for $n > k$ and only one $\T_k$-system while it has $\max(\varphi(d_1)/2, 1)$ Nielsen equivalence classes of generating $k$-vectors where $\varphi$ denotes the Euler totient function extended by $\varphi(0) = 2$.
\end{theorem}

In the remainder of this section, we consider decompositions with cyclic factors which 
might differ from the invariant factor decomposition of $G$.

\begin{corollary} \label{CorDet} \cite[Corollary C]{Guy17}
Let $G = \Z_{d_1} \times \cdots  \times \Z_{d_k}$ with $d_i \ge 0$ for $i = 1, \dots, k$. 
Let $d$ be the greatest common divisors of the integers $d_i$.
For $\gb \in \V_k(G)$, denote by $\det(\gb)$ the determinant of the matrix whose coefficients are the images in $\Z_d$ of the $(g_i)_j$'s 
via the natural maps $\Z_{d_j} \twoheadrightarrow \Z_d$.
Then $\gb, \gb' \in \V_k(G)$ are Nielsen equivalent if and only if $\det(\gb) = \pm \det(\gb')$.
\end{corollary}

Let $\Z_{d_1} \times \cdots  \times \Z_{d_k}$ be a decomposition of $G$ in cyclic factors such that $k = \rk(G)$. The identity elements of each factor ring form a generating vector of $G$. We refer to this vector as the \emph{generating vector naturally associated}  to the given decomposition. If $\eb$ is such a vector, we define $\det_{\eb}$ as the determinant function of Corollary \ref{CorDet}.

\begin{remark} \label{DetE}
In the case $k = \rk(G)$,
Corollary \ref{CorDet} shows in particular that $\gb \in \V_k(G)$ is Nielsen equivalent to $(\det_{\eb}(\gb)e_1, e_2, \dots, e_k)$ where $\eb = (e_i)$ is the generating vector naturally associated to the given decomposition in cyclic factors of $G$.
\end{remark}

\section{Nielsen equivalence classes and  $\T$-systems of $M \rtimes_{\alpha} C$}   \label{SecMC}
In this section, we prove Theorem \ref{ThMTimesC} and some prerequisites of Theorems \ref{ThN2} and \ref{ThT2}.
We assume throughout that $G$ is a group fitting in  the exact sequence (\ref{EqExt}) and $n \ge \max(\rk(G), 2)$. Recall that $\alpha$ denotes the image of a favored generator $a$ of $C$ via the natural map $\ZC \twoheadrightarrow R$.
Computing with powers of group elements in $G$ shall be facilitated by the following notation.
For $u \in R^{\times}$ and $l \in \Z \cup \{ \infty \}$, let 
$$
\partial_u(l)=\left\{
\begin{array}{cc}
1 + u + \cdots + u^{l - 1} & \text{ if } l > 0, \\
0 & \text{ if } l = 0, \,\infty, \\
- u^{-1} \partial_{u^{-1}}(-l) & \text{ if } l < 0.
\end{array}
\right.
$$
For every $l \in \Z$ we have then
$(1- u) \partial_u(l) = 1 - u^l$. In particular $(1 - \alpha)\nu(G) = 0$ for $\nu(G) = \partial_{\alpha}(\cC)$.
If $C$ is infinite then $\partial_{\alpha}$ is
the composition of the Fox derivative over $C$ \cite{Fox53} with the natural embedding of $C$ into $R = \ZC$. 
For $k,\,l \in \Z$ and  $m \in M$, we have the identity
$
(ma^k)^l = \left(\partial_{\alpha^k}(l)m\right) a^{kl}.
$

The description of Bachmuth's $\IA$ automorphisms will considerably ease off the study of Nielsen equivalence in $G = M \rtimes_{\alpha} C$.
Recall that $F_n$ denotes the free group on $\xb = (x_1,\dots,x_n)$. 
For $\psi \in \Aut(F_n)$, let $\ovpsi$ be the automorphism of $\Z^n$ induced by $\psi$. We denote by $\eb = (e_i)$ the image of $\xb$ under the abelianization homomorphism
$
F_n \twoheadrightarrow (F_n)_{ab}=\Z^n.
$
The map $\psi \mapsto \ovpsi$ is an epimorphism from $\Aut(F_n)$ onto $\GL_n( \Z)$ \cite[Proposition I.4.4]{LS77} 
whose kernel is denoted by $\IA(F_n)$. This group
clearly contains the isomorphisms $\varepsilon_{ij}$ defined by  
$\varepsilon_{ij}(x_i) = x_j^{-1} x_i x_j$ and  $\varepsilon_{ij}(x_k) = x_k$ if $k \neq i$.
 In turn, $\IA(F_n)$ is generated by the automorphisms $\varepsilon_{ijk}$ defined by 
$\varepsilon_{ijk}(x_i) = x_i \lbrack x_j, x_k \rbrack$ and $\varepsilon_{ijk}(x_l) = x_l$ for $l \neq i$  \cite[Chapter I.4]{LS77} where
$\br{x, y} \Doteq xyx^{-1}y^{-1}$.

\subsection{Reduction to an $a$-row} \label{SecARow}
We discuss here circumstances under which 
a generating $n$-vector $\gb$ of $G$ can be Nielsen reduced to an \emph{$a$-row}, i.e., 
a generating $n$-vector of the form $(\mb, a)$ with $\mb \in \U_{n - 1}(M)$ 
and $a$ a favored generator of $C$ fixed beforehand. We first observe 
that any generating $n$-vector $\gb$ is Nielsen equivalent to $(\mb, ma)$ 
for some $\mb \in  M^{n - 1}$ and some $m \in M$.
Indeed, we can find $\psi \in \Aut(F_n)$ 
such that $\sigma(\gb) \psi = (1_{C^{n - 1}}, a)$ using Theorem \ref{ThNielsenAbel}.
This proves the claim. We shall establish conditions under which the element $m$ can be cancelled by a subsequent Nielsen transformation.

\begin{lemma} \label{LemRankOfM}
 Let  $\gb=(\mb, ma) \in G^n$ with $\mb \in M^{n - 1}$ and $m \in M$. Then $\gb$ generates $G$ if and only if
$(\mb, \nu(G) m)$ generates $M$ as an $R$-module.
\end{lemma}
\begin{proof}
Assume first that $\gb \in \V_n(G)$.
Given $m' \in M$ there exists $w \in F_n$ such that $m' = w(\mb, ma)$. We can write $w = vx_n^s$ with $v$ lying in the normal closure of
$\{x_1, \dots, x_{n - 1}\}$ in $F_n$ and $s \in \Z$. Since conjugation by $ma$ induces multiplication by $\alpha$ on $M$, $v(\gb)$ lies in the $R$-submodule of $M$ generated by $\mb$. As $(m a)^s = (\partial_{\alpha}(s)m) a^s$, we deduce that $s = 0$ if $C$ is infinite or $s \equiv 0 \mod \cC$ if $C$ is finite. Therefore $(m a)^s$ belongs to $\Z \nu(G) m = R \nu(G) m$ in both cases, which completes the proof of the 'only if' part.

Assume now that $(\mb, \nu(G) m)$ generates $M$ as an $R$-module. Let $H$ be the subgroup of $G$ generated by $\gb$ and write $\mb = (m_i)$. The subgroup $H$ contains the conjugates of the elements $m_i$ by powers of $a$, hence it contains the submodule of $M$ generated by $\mb$. It also contains the powers of $m a$, hence the submodule of $M$ generated by $\nu(G) m$. Thus it contains both $M$ and $a$, so that it is equal to $G$.
\end{proof}

Lemma \ref{LemRankOfM} implies the following inequalities: $$\rk(G) - 1 \le \rk_R(M) \le \rk(G)$$ 
When every generating $n$-vector of $G$ can be Nielsen reduced to an $a$-row, we say that $G$ enjoys property $\Ni_n(a)$.
If $\Ni_n(a)$ holds for $n = \rk(G)$, then the equality $\rk_R(M) = \rk(G) - 1$ must be satisfied. 
The converse does not hold,  as the latter is equivalent to the weaker property $\Ni_n(C)$ according to which every generating $n$-vector can be Nielsen reduced to a $c$-row with $c$ ranging among generators of $C$, see Theorem \ref{ThReductionToARow} below. 

\begin{lemma} \label{LemGeneratorsOfM}
 Let $\gb=(\mb, ma) \in \V_n(G)$ with $\mb \in M^{n - 1}$ and $m \in M$. Then the following hold:
\begin{itemize}
\item[$(i)$]
If $m \in (1 - \alpha)M$ then $\mb$ generates $M$ as an $R$-module.
\item[$(ii)$]
If $\nu(G)$ is nilpotent then $\mb$ generates $M$ as an $R$-module.
\item[$(iii)$]
If $\mb$ generates $M$ as an $R$-module then $\gb$ is Nielsen equivalent to $(\mb, a)$.
\end{itemize}
\end{lemma}

\begin{proof}
We know that $(\mb, \nu(G) m)$ generates $M$ as an $R$-module 
by Lemma \ref{LemRankOfM}. If $m \in (1 - \alpha)M$ then $\nu(G) m = 0$. 
Hence $\mb$ generates $M$ as an $R$-module, which proves $(i)$. If $\nu(G) \in \Jac(R)$, 
the same conclusion follows from Nakayama's lemma, which proves $(ii)$.
Let us prove $(iii)$. If $\mb = (m_i)$ generates $M$ as an $R$-module 
then $m$ is a sum of elements of the form $k \alpha^l m'$ with $k,l \in \Z$ and 
$m' \in \{m_1, \dots, m_{n - 1}\}$. We can subtract each of these terms from $m$ 
in the last entry of $\gb$ by applying transformations of the form $\varepsilon_{i, n}^l$ and $L_{n, i}^{-k}$. 
\end{proof}
Combining assertions $(ii)$ and $(iii)$ of Lemma \ref{LemGeneratorsOfM} yields:
\begin{lemma} \label{LemNu}
If $\nu(G)$ is nilpotent then $\Ni_n(a)$ holds for every $n \ge \rk(G)$.
\end{lemma}

Let $\pi_{ab} :G \twoheadrightarrow \Gab$ be the abelianization homomorphism of $G$.
Let $\pi_C$ be the natural homomorphism $M \twoheadrightarrow M_{C} \Doteq M/(1-\alpha)M$. 
Then we have $\Gab = M_C \times C$ and $\pi_{ab} = \pi_C \times \sigma$.

\begin{proposition} \label{PropReductionToRow}
Let $\gb \in \V_n(G)$ and assume that at least one of the following holds:
\begin{itemize}
\item[$(i)$]
$n > \rk(\Gab)$,
\item[$(ii)$]
$n > \rk(M_C)$ and $M_C$ is not free over $\Z$.
\end{itemize}
 Then $\gb$ is Nielsen equivalent to a vector $(\mb, a)$ with $\mb \in \U_{n - 1}(M)$.
\end{proposition}

\begin{proof}
Let $k = \rk(M_C)$. Observe first that both assumptions imply $n > k$. Let $\Z_{d_1} \times \cdots \times \Z_{d_k}$ be 
the invariant factor decomposition of $M_C$. Let then 
$\Z_{d_1} \times  \cdots \times \Z_{d_{n -1}} \times C$ be the decomposition of  $\Gab$ where 
$d_i = 1$ if $i > k$.
Define $\eb = (e_i) \in \Gab^{n - 1}$ by $e_i = 1 \in \Z_{d_i}$ if $i \le k$, $e_i = 0$ otherwise.
 Set $\overline{\gb} = \pi_{ab}(\gb)$.
Suppose now that $(ii)$ holds so that $\Z_{d_1}$ must be finite. Let 
$\tilde{\delta} \in \Z_{d_1}^{\times}$ be a lift of $\delta \Doteq \det_{\eb}(\overline{\gb})$ and let 
$\tilde{\gb} = (\tilde{\delta}e_1, e_2, \dots, e_{n - 1}, a)$. Since 
$\det_{\eb}(\tilde{\gb}) = \delta = \det_{\eb}(\overline\gb)$, the vectors $\overline{\gb}$ and $\tilde{\gb}$ are Nielsen equivalent by Corollary \ref{CorDet}. 
By Theorem \ref{ThNielsenAbel} this is also true if we assume $(i)$ and set $\tilde{\delta} = 1 \in \Z_{d_1}$.
Hence under assumption $(i)$ or $(ii)$ there is $\psi \in \Aut(F_n)$ such that $\tilde{\gb} = \overline{\gb}  \psi$. 
Then $\gb' \Doteq \gb \psi$ is of the form $(\mb, ma)$  with $\mb = (m_1,\dots, m_{n - 1}) \in  M^{n - 1}$ and $m \in M$ such that $\pi_C(m) = 0$. 
Applying Lemma \ref{LemGeneratorsOfM} gives the result.
\end{proof}

\begin{proposition} \label{PropReductionToRowFree}
Assume $C$ is finite and let $n = \rk(G)$. Suppose moreover that $M_C$ is isomorphic to $\Z^{n - 1}$. 
Let $\gb \in \V_n(G)$. Then $\gb$ is Nielsen equivalent to a vector $(\mb, a^k)$ with $\mb \in \U_{n - 1}( M)$ and 
$k = \pm \det_{(\eb, a)}(\ov{\gb})$ where $\ov{\gb} = \pi_{ab}(\gb)$ and $\eb$ is any basis of $M_C$. In particular, $\rk_R(M) = \rk(G) - 1$.
\end{proposition} 
\begin{proof}
Let $\mathbf{e}$ be a basis of $M_C$ over $\Z$. By Theorem \ref{ThNielsenAbel}, we can assume that
$\ov{\gb} = (\mathbf{e}, a^k)$ for some $k$ such that
$k = \pm \det_{(\eb, a)}(\overline{\gb})$. 
Hence $\gb = (\mb, ma^k)$ for some $\mb \in M^{n -1}$ and $m \in (1 - \alpha)M$. 
Then $\gb$ is Nielsen equivalent to $(\mb, a^k)$ by Lemma \ref{LemGeneratorsOfM}.
\end{proof}

Eventually, we present a characterization of property $\Ni_n(a)$ which establishes the first part of Theorem \ref{ThN2}.

\begin{theorem} \label{ThReductionToARow}
Property $\Ni_n(a)$ holds if, and only if,  at least one of the following holds:
\begin{itemize} 
\item[$(i)$] $n  > \rk(\Gab)$.
\item[$(ii)$] $C$ is infinite.
\item[$(iii)$] $\rk(G) > \rk(M_C)$ and $M_C$ is not isomorphic to $\Z^{\rk(G) - 1}$.
\item[$(iv)$] $\cC \in \{2, 3, 4, 6\}$ and $M_C$ is isomorphic to $\Z^{\rk(G) - 1}$.
\end{itemize}
\end{theorem} 

\begin{proof}
Let us show first that any of the assertions $(i)$ to $(iv)$ implies $\Ni_n(a)$.
For assertion $(i)$, it is Proposition \ref{PropReductionToRow}. For assertion $(ii)$, it follows from Lemma \ref{LemNu}. For the remaining assertions, we can assume that assertion $(i)$ doesn't hold, so that $n = \rk(G)$.
Then assertion $(iii)$ implies $\Ni_n(a)$, by Proposition \ref{PropReductionToRow}. So does assertion $(iv)$ by Proposition \ref{PropReductionToRowFree}.

Let us assume now that none of the assertions $(i)$ to $(iv)$ hold. We shall prove that property $\Ni_n(a)$ doesn't hold.
Since assertion $(i)$ is assumed not to hold, we infer that $n = \rk(G)$. We can assume moreover that $\rk_R(M) = n - 1$, since otherwise $\Ni_n(a)$ would fail to be true. Because none of the assertions $(ii)$, $(iii)$ and $(iv)$ hold, the group $M_C$ is isomorphic to $\Z^{\rk(G) -1}$ and $C$ is finite and such that $\varphi(\cC) > 2$. By Proposition \ref{PropReductionToRowFree}, we can find $\mb \in \U_{n - 1}(M)$ and $k > 1$ coprime with $\cC$ such that $(\mb, a^k)$ generates $G$ but cannot be Nielsen reduced to an $a$-row. 
\end{proof}

\begin{corollary} \label{CorNC}
Property $\Ni_n(C)$ holds if and only if $n > \rk_R(M)$.
\end{corollary} 
\begin{proof}
If $\Ni_n(C)$ holds, the inequality $\rk_R(M) \le n - 1$ is satisfied by definition. Let us assume now that $\Ni_n(C)$ doesn't hold. 
Reasoning by contradiction, we assume furthermore that $n > \rk_R(M)$.
Since property $\Ni_n(a)$ cannot hold it follows from Theorem \ref{ThReductionToARow} that $n = \rk(G)$, $C$ is finite and $M_C$ is isomorphic to $\Z^{\rk(G) -1}$. The latter three conditions imply $\Ni_n(C)$ by Proposition \ref{PropReductionToRowFree},  a contradiction.
\end{proof}

\begin{corollary} \label{CorNc}
Property $\Ni_n(c)$ holds for some generator $c$ of $C$ if and only if it holds for all generators $c$ of $C$.
\end{corollary} 
$\square$

\subsection{Nielsen equivalence related to $\Gamma_{n - 1}(R)$-equivalence} \label{SecCIsZ}
In this section we scrutinize the relation between Nielsen equivalence of generating $n$-vectors and $\Gamma_{n - 1}(R)$-equivalence of unimodular rows. We prove here another part of Theorem \ref{ThMTimesC}, namely Proposition \ref{PropConverseZ} below.

Recall that $T$ denotes the subgroup of $R^{\times}$ generated by $-1$ and $\alpha$ and that $\Gamma_n(R)$ is 
the subgroup of $\GL_n(R)$ generated by $\El_n(R)$ and $\D_n(T)$.
Since $\D_n(R^{\times})$ normalizes $\E_n(R)$, we have
$
\Gamma_n(R) = \D_n(T) \El_n(R)
$.

\begin{lemma} \label{LemGeneratorsOfGE}
For every $n \ge 2$, the group
$\Gamma_n(R)$ is generated by $\D_n(T)$ together with the elementary matrices $E_{ij}(1)$ with $1 \le i \neq j \le n$. 
\end{lemma}

\begin{proof}
 For $1 \le i \neq j \le n$ and $(r,r') \in R^2, \, \beta \in \{\alpha^{\pm1}\}$, we have the following identities: 
 $D_i(\beta)E_{ij}(r)D_i(\beta)^{-1} = E_{ij}(\beta r)$ and
$E_{ij}(r)E_{ij}(r') = E_{ij}(r + r')$.
Since $R$ is generated as a ring by $\alpha$ and $\alpha^{-1}$, the result follows.  
\end{proof}

\begin{lemma} \label{LemGammaImpliesNielsen}
If $\mb, \mb' \in \U_{n - 1}( M)$ are  $\Gamma_{n - 1}(R)$-equivalent, then 
$(\mb,\, a)$, $(\mb', \,a) \in \V_n(G)$ are Nielsen equivalent.
\end{lemma}

\begin{proof} Since $(\mb E_{ij}(1), a) = (\mb, a) L_{ij}$ for $1 \le i \neq j \le n -1$ and 
$(\mb D_{i}(\alpha), a) = (\mb, a) \varepsilon_{i, n}$ for 
$1 \le i \neq j \le n - 1$, we deduce from Lemma \ref{LemGeneratorsOfGE} that 
$(\mb, a)$ and $(\mb', a)$ are Nielsen equivalent.
\end{proof}

We establish now a partial converse to Lemma \ref{LemGammaImpliesNielsen}.

\begin{proposition} \label{PropConverseZ}
Assume $C$ is infinite. If  $(\mb, a), (\mb', a) \in \V_n(G)$ are Nielsen equivalent then 
$\mb, \mb' \in \U_{n - 1}(M)$ are  $\Gamma_{n - 1}(R)$-equivalent.
\end{proposition}

\begin{proof}
Suppose that $(\mb', a) = (\mb, a) \psi$ for some $\psi \in \Aut(F_n)$. 
We claim that $\psi$ is of the form $\psi_0 \psi_1 L$ where
\begin{itemize}
\item $\psi_0 \in \IA(F_n)$, 
\item $\psi_1 \in \Aut(F_{n - 1})$, i.e., $\psi_1(x_n) = x_n$ and $\psi_1$ 
leaves $F_{n - 1}=F(x_1, \dots, x_{n - 1})$ invariant,
\item $L$ belongs to the group generated by the automorphisms $L_{n, j}$.
\end{itemize}

To see this, consider the automorphism $\ovpsi \in \GL_n(\Z)$ induced by $\psi$. 
Since $\ovpsi(e_n) = e_n$, 
we can find a product of lower elementary matrices 
$$\overline{L} \Doteq E_{n, 1}(\mu_1) \cdots E_{n , n -1}(\mu_{n - 1})
$$ with $\mu_i \in \Z$ 
such that 
$\ovpsi  \cdot \overline{L}^{-1} \in \GL_{n - 1}(\Z)$. 
Let $\psi_1 \in \Aut(F_{n - 1})$ be an automorphism inducing $\ovpsi \cdot \overline{L}^{-1}$ on $\Z^{n - 1}$. Let $L$ be the product of automorphisms $L_{n,  j}^{\mu_j}$ with $\mu_j \in \Z$. 
Then $L$ induces $\overline{L}$ and by construction 
we have $\psi L^{-1} \psi_1^{-1} \in \IA(F_n)$, which proves the claim. 

The action of every $\IA$-automorphism $\varepsilon_{ijk}$ on $(\mb, a)$ leaves $\sigma(\mb, a)$ 
invariant and induces a transformation on $\mb$ which lies in $\Gamma_{n - 1}(R)$. 
The same holds for every automorphism in $\Aut(F_{n - 1})$ and every automorphism $L_{ij}$ with $i>j$. 
Let $\gb \Doteq (\mb, a) \psi_0 \psi_1$. Then we can write $\gb = (\mathbf{n}, ma)$ with $m \in M$ and where $\mathbf{n}$ is 
$\Gamma_{n - 1}(R)$-equivalent to $\mb$. As $\sigma(\gb L) = \sigma(\mb', a) = (1_{C^{n -1}}, a)$ we deduce that $\mu_j = 0$ 
for every $j$, i.e., $L = 1$. Hence $\mb' = \mathbf{n}$ which yields the result.
\end{proof}

\subsection{Nielsen equivalence classes} \label{SecNielsen}
In this section, we complete the proof of Theorem \ref{ThMTimesC} by establishing Theorem \ref{ThNielsenCInfinite} below. We subsequently discuss assumptions under which the latter theorem enables us to enumerate efficiently Nielsen equivalence classes.
Recall that the map $\varphi_a: \U_{n - 1}(M) \rightarrow \V_n(G)$ is defined by
$\varphi_a(\mb) = (\mb, a)$.
\begin{theorem} \label{ThNielsenCInfinite}
The map $\varphi_a$ induces a map 
$$
\Phi_a: \U_{n - 1}(M)/\Gamma_{n - 1}(R) \rightarrow \V_n(G)/\Aut(F_n)
$$
and the two following hold
\begin{itemize}
\item[$(i)$]
Property $\Ni_n(a)$ holds if and only if $n > \rk_R(M)$ and $\Phi_a$ surjective.
\item[$(ii)$]
If $C$ is infinite then $\Phi_a$ is bijective.
\end{itemize}
\end{theorem}

\begin{proof}
It follows from Lemma \ref{LemGammaImpliesNielsen} that $\Phi_a$ is well-defined. Assertion $(i)$ is trivial 
while assertion $(ii)$ results from Lemma \ref{LemGeneratorsOfM}.$ii$ and Proposition
\ref{PropConverseZ}.
\end{proof}

\begin{corollary} \label{CorStableRank}
Assume that $M \simeq R$.
If $n >\sr(R) + 1$, 
then $G = M \rtimes_{\alpha} C$ has only one Nielsen equivalence class of generating $n$-vectors. 
\end{corollary}
\begin{proof}
The result follows from Theorem \ref{ThMTimesC} and the inequality (\ref{EqStableRank}).
\end{proof}

We consider now several hypotheses under which the problem of counting Nielsen equivalence classes is particularly tractable.
One of these hypotheses is that $R$ be \emph{quasi-Euclidean}, i.e.,  $R$ enjoys the following row reduction property shared by Euclidean rings:
for every $n \ge 2$ and every $\rb = (r_1, \dots, r_n) \in R^n$, there exist $E \in \El_n(R)$ and $d \in R$ such that   
$
(d, 0,\dots,0) = \rb E
$
 (see \cite{AJLL14} for a comprehensive survey on quasi-Euclidean rings). If $R$ is a Noetherian quasi-Euclidean ring,
 then $M$ admits an \emph{invariant factor decomposition}, i.e., a decomposition of the form 
$R /\ia_1 \times R/\ia_2 \times \cdots \times R /\ia_n$ with 
$R \neq \ia_1 \supset \ia_2 \supset \cdots \supset \ia_n$ where the ideals $\ia_i$ are referred to as the \emph{invariant factors} of $M$ (see \cite[Lemma 1]{Guy17}).
Recall that we denote by $T$ the subgroup of $R^{\times}$ generated by $-1$ and $\alpha$

\begin{corollary} \label{CorGE}
Let $G = M \rtimes_{\alpha} C$ and $n = \rk(G)$. Then the following hold:
\begin{itemize}
\item[$(i)$]
If $M$ is free over $R$, $C$ is infinite and $R$ is a $\GE_{n - 1}$-ring, then \\$\nielsen_n(G) = \vert R^{\times}/T \vert$.
\item[$(ii)$]
If $C$ is infinite and $R$ is quasi-Euclidean, then $\nielsen_n(G) = \vert \Lambda^{\times}/T_{\Lambda} \vert$ where 
$\Lambda = R/\ia_1$, $\ia_1$ is the first invariant factor of $M$ and $T_{\Lambda}$ is the image of $T$ in $\Lambda$ under the natural map.
\end{itemize}
\end{corollary}
\begin{proof}
$(i)$. As $C$ is infinite, it follows from Theorem \ref{ThMTimesC} that $M \simeq R^{n - 1}$. For $\mb \in \U_{n - 1}(M)$, let 
$\Mat(\mb) \in \GL_{n - 1}(R)$ be the matrix whose columns are the components of $\mb$. For every $E \in \Gamma_{n - 1}(R)$, 
the identity $\Mat(\mb E) = \Mat(\mb)E$ holds. As $R$ is a $\GE_{n - 1}$-ring, we have $\Gamma_{n - 1}(R) = \D_{n -1}(T) \SL_{n - 1}(R)$. We deduce from 
Whitehead's lemma that $\Mat(\mb)$ can be reduced to $D_{n - 1}(u)$ via right multiplication by some $E \in \Gamma_{n -1}(R)$, where $u$ is a member of a transversal of $R^{\times}/T$. 
Therefore $\nielsen_n(G) \le  \vert R^{\times}/T \vert$.
Since $uT = \det(\Mat(\mb)E)T$, we conclude that $\nielsen_n(G) = \vert R^{\times}/T \vert$.

$(ii)$. By \cite[Theorem A and Corollary C]{Guy17},  we have $\U_n(M)/\Gamma_n(R) \simeq \Lambda^{\times}/T_{\Lambda}$. Theorem \ref{ThMTimesC} implies that $\nielsen_n(G) = \left\vert\Lambda^{\times}/T_{\Lambda}\right\vert$.
\end{proof}

We examine in the next proposition the structural implication of $M$ being free over $R$ with $R$-rank equal to $\rk(G)$.
\begin{proposition} \label{PropCyclicDecMaxRank}
Assume that $\rk_R(M) = n = \rk(G)$ and that $M$ is the direct sum of $n$ cyclic factors,  i.e.,
$$
M = R /\ia_1 \times  \cdots \times R /\ia_n
$$
where the $\ia_i$ are ideals of $R$.
Let $\ia = \ia_1 + \cdots + \ia_n$.
Then $\nu(G)$ is invertible modulo $\ia$. In addition, $C$ is finite and $G/ \ia M = \Z_d^{\rk(G)} \times C$ where 
$d = \vert R/\ia \vert < \infty$ is prime to $\cC$.
\end{proposition}

\begin{proof}
We can assume without loss of generality that $\ia = \{0\}$. 
Let $\eb = (e_i)$ be a basis of $M$ over $R$ and let $\gb \in \V_n(G)$ for $n = \rk(G)$. 
Replacing $\gb$ by $\gb \psi$ for some $\psi \in \Aut(F_n)$, if needed, we can also suppose that $\gb = (\mb, ma)$ with $\mb \in M^{n - 1}$ and $m \in M$. 
By Lemma \ref{LemRankOfM}, the row $(\mb, \nu(G) m)$ generates $M$ as an $R$-module. 
Therefore the map $\eb \mapsto (\mb, \nu(G) m)$
 induces an $R$-automorphism of $M$. 
This shows that $e_n = \nu(G) m'$ for some $m' \in M$. Hence a relation of the form
$\sum_{i = 1}^{n - 1} r_i e_i + (\nu(G) r_n - 1) e_n = 0$, with $r_i \in R$ holds in $M$. It follows that $\nu(G)$ is invertible. 
Thus $M = \nu(G) M$ is $C$-invariant so that $G = M_C \times C$. 
As a result $M = M_C$ is a free $\Z_d$-module with 
$d = \vert R \vert$ or $d = 0$. Since $\rk(M) = \rk(G)$, the group $C$ must be finite, $d$ must be non-zero and prime to $\cC$. 
\end{proof}

\begin{corollary} \label{CorGEMFree}
Let $G = M \rtimes_{\alpha} C$. Assume that at least one of the following holds:
\begin{itemize}
\item[$(i)$]
$R$ is quasi-Euclidean.
\item[$(ii)$]
$M$ is free over $R$ and $R$ is a $\GE$-ring
\end{itemize} 
Then $\nielsen_n(G) = 1$ for every $n > \rk(G)$.
\end{corollary}

\begin{proof}
Suppose first that $(i)$ holds. By \cite[Corollary B]{Guy17}, the set \\
$\U_{n - 1}(M)/\Gamma_{n - 1}(R)$ is reduced to one element and Theorem \ref{ThMTimesC} implies $\nielsen_n(G) = 1$.
Suppose now that $(ii)$ holds. If $\rk_R(M) = \rk(G)$, then $G$ is Abelian by Proposition \ref{PropCyclicDecMaxRank} and the result follows from Theorem \ref{ThNielsenAbel}. Assume now that $M \simeq R^{\rk(G) - 1}$, let $k = \rk(G)$ and $n > k$. 
As the result certainly holds if $G$ is cyclic, we can assume moreover that $k \ge 2$.
Let us show that $\U_{n - 1}(M)/\E_{n - 1}(R)$ is made of a single orbit.
For $\mb \in \U_{n - 1}(M)$, let 
$\Mat(\mb)$ be the $(k -1)$-by-$(n -1)$ matrix whose columns are the components of $\mb$.
Since $R$ is a $\GE$-ring, there is $E \in \E_{n - 1}(R)$ such that 
\begin{equation} \label{EqMatM}
\Mat(\mb)E = 
\begin{pmatrix} 1 & 0 \\  A & B \end{pmatrix}
\end{equation}
where $A$ is a $(k - 2)$-by-$1$ matrix and $B$ is a $(k - 2)$-by-$(n -2)$ matrix.
If $k = 2$, then $\mb$ has been reduced to a standard unimodular row. Otherwise, let  $\mb' = \mb E$. Since $\mb'$ generates $M$, we can find a column vector $V = \begin{pmatrix} v_1 \\ \vdots \\ v_{n - 1} \end{pmatrix}$ such that $\Mat(\mb')V = \begin{pmatrix} 1 \\ 0 \end{pmatrix}$. Combining the latter identity with (\ref{EqMatM}), we deduce that $v_1 = 1$ and subsequently 
$\Mat(\mb')P = 
\begin{pmatrix} 1 & 0 \\  0 & B \end{pmatrix}$ where $P = E_{1,2}(v_2) \cdots E_{1, n -1}(v_{n - 1})$. By iterating this procedure, we reduce 
$\Mat(\mb)$ to $\begin{pmatrix} I_{k - 1} & 0 \end{pmatrix}$ through elementary column reduction operations. Consequently, 
$\U_{n -1}(M)/\E_{n - 1}$ contains a single orbit. Theorem \ref{ThMTimesC} eventually implies that $\nielsen_n(G) = 1$.
\end{proof}

\subsection{T-systems} \label{SecTsys}
In this section, we prove results on the $\T$-systems of $G = M \rtimes_{\alpha} C$ under the assumption that $M$ is free over $R$.
 These results will be specialized in Section \ref{SecT2} so as to prove Theorem \ref{ThT2}. We present first all the definitions needed for describing $\Aut(G)$.
Let $c \mapsto \ov{c}$ be the restriction to $C$ of the natural map $\ZC \twoheadrightarrow R$.
We call $d:C \longrightarrow M$ a \emph{derivation} if $d(cc') = d(c) + \overline{c}d(c')$  holds for every $(c,c') \in C^2$.
Given a derivation $d$, we denote by $X_d$ the automorphism of $G$ defined by 
$$mc \mapsto m d(c) c$$
For $t \in \Aut_R(M)$, denote by $Y_t$ the automorphism of $G$ defined by 
$$mc \mapsto t(m)c$$

The following lemma underlines the link between $\Der(C, M)$, the $R$-module of derivations, and the automorphisms of $G$ which leave $M$ point-wise invariant.
\begin{lemma} \label{LemDerivation}
Let $m \in M$. Then the following are equivalent:
\begin{itemize}
\item[$(i)$] $\nu(G) m = 0$.
\item[$(ii)$] There exists $d \in \Der(C, M)$ such that $d(a) = m$.
\item[$(iii)$] There exists $\phi \in \Aut(G)$ such that $\phi(a) = ma$.
\end{itemize} 
If one of the above holds, then the derivation $d$ in $(ii)$ is uniquely defined by $d(a^k) = \partial_{\alpha}(k)m$ for every $k \in \Z$.
If in addition the restriction to $M$ of the automorphism $\phi$ in $(iii)$ is the identity, then $\phi(mc) = md(c)c$ for every $(m , c) \in M \times C$.
\end{lemma}
$\square$

We denote by $A(C)$ the subgroup of the automorphisms of $C$ induced by automorphisms of $G$ preserving $M$.
The following result is referred to in \cite[Proposition 4]{Guy12} where it is a key preliminary to the study of two-generated  $G$-limits in the space of marked groups.
\begin{proposition} \label{PropGenAut}
Assume $M$ is a free $R$-module and $\nu(G) = 0$.
Let $n = \rk(G)$ and let $\gb, \gb' \in \V_n(G)$.
Then the following are equivalent:
\item[$(i)$] $\gb$ and $\gb'$ are related by an automorphism of $G$ preserving $M$,
\item[$(ii)$] $\sigma(\gb)$ and $\sigma(\gb')$ are related by an automorphism in $A(C)$.
\end{proposition}

\begin{proof}
Clearly, assertion $(i)$ implies $(ii)$. Let us prove the converse. Replacing, if needed, $\gb$ by $\phi \gb$ for some automorphism $\phi$ of $G$ that preserves $M$, we can assume without loss of generality that 
$\sigma(\gb) = \sigma(\gb')$. 
Replacing $\gb$ and $\gb'$ by $\gb \psi$ and $\gb' \psi$ respectively for some $\psi \in \Aut(F_n)$, we can also assume that 
$\sigma(\gb) = \sigma(\gb') = (1_{C^{n - 1}}, a)$ 
and hence write $\gb = (\mb, ma)$ and $\gb' = (\mb', m' a)$ with $\mb, \mb' \in M^{n - 1}$ and $m, m' \in M$. 
By Lemma \ref{LemGeneratorsOfM}, the rows $\mb$ and $\mb'$ are bases of $M$.
By Lemma \ref{LemDerivation}, there is $d' \in \Der(C, M)$ such that $d'(a) = m - m'$. Let $t \in \Aut_R(M)$ be defined by $t(\mb)=\mb'$. Then we have 
$\gb = Y_tX_{d'} \gb'$, which proves the result.
\end{proof}

Recall that $G$ is said to have property $\Ni_n(a)$ if every of its generating $n$-vectors can be Nielsen reduced to an $a$-row, i.e., a vector of the form $(\mb, a)$ with $\mb \in \U_{n -1}(M)$.

\begin{theorem} \label{ThOneTSystem}
Let $n = \rk(G)$ and assume that $M$ is free over $R$. If moreover $\Ni_n(a)$ holds or $\rk_R(M) = n$,
then $G$ has only one $\T_n$-system.
\end{theorem}

\begin{proof} 
Assume $\Ni_n(a)$ holds.
Given an $R$-basis $\mathbf{e}$ of $M$, we shall prove that any $\gb \in \V_n(G)$ is in the same $\T_n$-system as 
$(\eb, a)$. Since $\Ni_n(a)$ holds, we can assume without loss of generality that $\gb$ is of the form 
$(\mb, a)$ with  $\mb \in \U_{n -1}(M)$.
The $R$-endomorphism $t$ mapping $\mathbf{e}$ to $\mb$ is an $R$-isomorphism.
Thus $(\mb, a) = Y_t(\eb, a)$.
Let us assume now that $\rk_R(M) = n$ holds. The group $G$ is then a finite Abelian group by Proposition \ref{PropCyclicDecMaxRank}. 
Thus the result follows from Theorem \ref{ThNielsenAbel}.
\end{proof}

We denote by $A'(C)$ the subgroup of $\Aut(C)$ generated by $A(C)$ and the automorphism of $C$ which maps $a$ to $a^{-1}$.
Here is our most general result regarding the number of $\T$-systems.
\begin{theorem} \label{ThTSystems}
Let $n = \rk(G)$ and assume that $M$ is free over $R$.
Then $\tsys_n(G) \le \vert \Aut(C) / A'(C) \vert$, with equality if $M$ is a characteristic subgroup of 
$G$ and $M_C$ is isomorphic to $\Z^{n - 1}$.
\end{theorem}

In order to prove the above theorem, we will use this simple variation on results found in \cite{Sze04}.
\begin{lemma} \label{LemAutG}
Let $\Aut_{\Z}(M)$ be the group of $\Z$-automorphisms of $M$.
\begin{itemize}
\item[$(i)$]
Let $(\tau, \theta) \in \Aut_{\Z}(M) \times \Aut(C)$ such that $\theta(a) = a^k$ and
$$\tau(cmc^{-1}) = \theta(c)\tau(m)\theta(c)^{-1} \text{ for all } (m, c) \in M \times C.$$
Then $Y_{\tau, \theta}: mc \mapsto \tau(m) \theta(c)$ is an automorphism of $G$ 
and the map $\alpha \mapsto \alpha^k$ induces a ring automorphism $\overline{\theta}$ of $R$ which satisfies
$$\tau(rm) = \overline{\theta}(r)\tau(m) \text{ for all } (r, m) \in R \times M.$$
\item[$(ii)$]
Let $d \in \Der(C, M)$, $t \in \Aut_R(M)$ and $(\tau, \theta)$ as in $(i)$. Then we have $\tau^{-1} \circ d \circ \theta \in \Der(C, M)$, 
$\tau^{-1} \circ t \circ \tau \in \Aut_R(M)$ and
$$Y_{\tau, \theta} X_d Y_{\tau, \theta}^{-1} = X_{\tau^{-1} \circ d \circ \theta}, \quad
Y_{\tau, \theta} Y_t Y_{\tau, \theta}^{-1} = Y_{\tau^{-1} \circ t \circ \tau}.$$
\item[$(iii)$]
Let $\phi \in \Aut(G)$ such that $\phi(M) = M$. Then there is $d \in \Der(C, M)$ and $(\tau, \theta)$ as in $(i)$, such that 
$\phi = X_d Y_{\tau, \theta}$. In particular, every automorphism in $A(C)$ is induced by some $Y_{\tau, \theta} \in \Aut(G)$.
\end{itemize}
\end{lemma}

\begin{proof}
The proofs of assertions $(i)$ and $(ii)$ are straightforward verifications. 

$(iii)$. Let $\tau$ be the restriction of $\phi$ to $M$ and let $\theta$ the automorphism of $C$ induced by 
$\phi$. It is easy to check that $(\tau, \theta)$ satisfy the conditions of $(i)$. Let $\phi' = \phi Y_{\tau, \theta}^{-1}$. Then the restriction of $\phi'$ to $M$ is the identity and there is $m \in M$ such that $\phi'(a) = ma$. By Lemma \ref{LemDerivation}, there is $d \in \Der(C, M)$ such that $\phi' = X_d$. 
\end{proof}

\begin{sproof}{Proof of Theorem \ref{ThTSystems}}
If the property $\Ni_a(n)$ holds, or if $\rk_R(M) = n$, then Theorem \ref{ThOneTSystem} implies that $\tsys_n(G) = 1 \le \vert \Aut(C) / A'(C) \vert$. Therefore we can assume, without loss of generality that $\Ni_a(n)$ does not hold and $\rk_R(M) < n$.  

By Theorem \ref{ThMTimesC}, the group $C$ is finite and $M_C$ is isomorphic to $\Z^{n - 1}$.
Using Proposition \ref{PropReductionToRowFree} and reasoning with a basis $\eb$ of $M$ as in the proof of Theorem \ref{ThOneTSystem}.$i$, we see that every generating $n$-vector falls into the $\T_n$-system of $(\eb, a^k)$ for some $k$ coprime with $\cC$. It follows from Lemma \ref{LemAutG}.$ii$ that $(\eb, \theta(a^k))$ 
lies in the $\T_n$-system of $(\eb, a^k)$ for every $\theta \in A'(C)$. Therefore $\tsys_n(G) \le \vert \Aut(C) / A'(C) \vert$.
Assume now that $M$ is a characteristic subgroup of $G$. If $(\eb, a)$ lies in the $\T_n$-system of 
$(\eb, a^k)$ for some $k$ coprime with $\cC$, then we can find $\phi \in \Aut(G)$ such that $\phi(\eb, a)$ is Nielsen equivalent to 
$(\eb, a^k)$. By Lemma \ref{LemAutG}.$ii$, we have 
$\phi(\eb, a) = (\eb', m \theta(a))$ for some basis $\eb'$ of $M$ and some $(m, \theta ) \in M \times A(C)$. 
By Lemma \ref{LemGeneratorsOfM}.$iii$, the vector $(\eb', \theta(a))$ is Nielsen equivalent to $(\eb, a^k)$. Proposition \ref{PropReductionToRowFree} implies that $\theta(a) = a^{\pm k}$, hence there is $\theta' \in A'(C)$ such that $\theta'(a) = a^k$.
\end{sproof}

\section{Nielsen equivalence classes and $\T$-systems of $R \rtimes C$} \label{SecRC}
In this section, we assume that $M \simeq R$, i.e., $G=\langle a,b\rangle$ is a split extension of the form $R \rtimes_{\alpha} C$ with $C = \langle a \rangle$, while $b$ is the identity of the ring $R$ and $a$ acts on $R$ as the multiplication by $\alpha \in R^{\times}$. As usual, $T$ denotes the subgroup of $R^{\times}$ generated by $-1$ and $\alpha$.
\subsection{Nielsen equivalence of generating pairs} \label{SecNielsenPairs}
We prove here the first two assertions of Theorem \ref{ThN2}. We begin with the definition of an invariant of Nielsen equivalence named
$\Delta_a$. Recall that $\pi_{\nu(G) R}$ denotes the natural group homomorphism $R \rtimes C \twoheadrightarrow (R/\nu(G) R) \rtimes C$ as well as the induced map on generating pairs.
If $\nu(G) = 0$, there is a unique derivation $d_a \in \Der(C, R)$ satisfying $d_a(a)=1$.  
For $\gb = (g, g') = (rc, r'c') \in G^2$ with $(r,r') \in R^2$ and $(c,c') \in C^2$, we set
\begin{equation} \label{EqDa}
D_a(\gb) = r d_a(c') - r' d_a(c) \in R
\end{equation}
 It is easily checked that
$\br{g, g'}=(1 -\alpha)D_a(\gb).$
If $\nu(G) \neq 0$, we set further
$$D_a(\gb) = D_a(\pi_{\nu(G) R}(\gb)) \in R/\nu(G) R$$ where the right-hand side is defined as in (\ref{EqDa}).

\begin{lemma} \label{LemmaDelta}
We have $D_a(\gb) \in (R / \nu(G) R)^{\times}$ for every $\gb \in \V_2(G)$. 
\end{lemma}

\begin{proof}
We can assume, without loss of generality, that $\nu(G) = 0$.\\
Let $\gb = (ra^k, r'a^{k'}) \in \V_2(G)$ with
$(r,r') \in R^2$ and $(k,k') \in \Z^2$. 
We first observe that
$$ 
\begin{array}{ccc}
D_a(\gb L_{12}) &=& \alpha^{k'} D_a(\gb)\\
D_a(\gb L_{21}) &=& \alpha^{k} D_a(\gb)\\
D_a(\gb I_1) &=&  - \alpha^{-k}D_a(\gb)
\end{array}
$$
Thus $D_a(\gb \Aut(F_2)) = T D_a(\gb)$. 
We know from Lemma \ref{LemGeneratorsOfM} that $\gb$ is Nielsen equivalent to $(r, a)$ for some $r \in R^{\times}$. 
Therefore $D_a(\gb) \in rT$, which shows that $D_a(\gb)$ is invertible.
\end{proof}

\begin{remark} \label{RemarkDaUnicity}
Assume $\nu(G) = 0$. Let $c$ be a generator of $C$ and let $d_c \in \Der(C, R)$ such that $d_c(c) = 1$. 
It is easily checked that 
$d_c = d_c(a) d_a$ and the identity 
$d_c(c) = 1$ implies that $d_c(a) \in R^{\times}$. For such elements $c$ there is thus only one map $D_c$ up to multiplication by a unit of $R$.
\end{remark}

We set $$\Lambda = R / \nu(G) R, \quad T_{\Lambda} = \pi_{\nu(G) R}(T),$$ 
and define the map
$$
\begin{array}{cccc}
\Delta_a:&  \V_2(G) & \rightarrow & \Lambda^{\times}/ T_{\Lambda} \\
           & \gb         & \mapsto                & T_{\Lambda} D(\gb)
\end{array}
$$

In the course of Lemma \ref{LemmaDelta}'s proof we actually showed
\begin{lemma} \label{LemDeltaAInvariant}
The map $\Delta_a$ is $\Aut(F_2)$-invariant.
\end{lemma}

Here is the last stepping stone to the theorem of this section.
\begin{lemma} \label{LemReductionModuloNuRank2}
Let $c$ be such that $C = \langle c \rangle$ and let $\gb = (r, c)$, $\gb' = (r', c)$ with $(r, r') \in (R^{\times})^2$. Then the following are equivalent:
\begin{itemize}
\item[$(i)$] $\gb$ and $\gb'$ are Nielsen equivalent.
\item[$(ii)$]  $\pi_{\nu(G) R}(\gb)$ and $\pi_{\nu(G) R}(\gb')$ are Nielsen equivalent.
\item[$(iii)$] $\Delta_a(\gb) = \Delta_a(\gb')$.
\end{itemize}
\end{lemma}

\begin{proof}
$(i) \Rightarrow (ii)$. This follows from the $\Aut(F_2)$-equivariance of $\pi_{\nu(G) R}$.

$(ii) \Rightarrow (iii)$. This follows from Remark \ref{RemarkDaUnicity} and Lemma \ref{LemDeltaAInvariant}.

$(iii) \Rightarrow (i)$.
The result is trivial if $\nu(G) = 0$, thus we can assume that $C$ is finite.
By hypothesis, there exist $k \in \Z$, $r_{\nu} \in R$ and $\epsilon \in \{\pm 1\}$ such that 
$r' = \epsilon \alpha^kr + r_{\nu}\nu(G)$. Replacing $\gb'$ by a conjugate if needed, we can assume that $k = 0$. 
Taking the inverse of the first component of $\gb$ if needed, we can moreover assume that $\epsilon = 1$, so that $r' = r + \nu(G) r_{\nu}$. 
Since $r' $ is a unit, we can argue as in the proof of Lemma \ref{LemRankOfM}.$iii$ to get $\psi \in \Aut(F_2)$ such that 
$(r', c) \psi =  (r', r_{\nu} c)$. We have then $(r', c) \psi L_{1, 2}^{-\cC} = (r, r_{\nu} c)$. Since $r$ is a unit, we can cancel $r_{\nu}$ using another automorphism of $F_2$.
\end{proof}

The next lemma will help us determine when $\Delta_a$ is a complete invariant of Nielsen equivalence.

\begin{lemma}[\cite{mathoverflow}] \label{LemSurjectivity}
Let $I \subset R$ be an ideal which is contained in all but finitely many maximal ideals of $R$. 
Then the natural map $R^{\times} \rightarrow (R/I)^{\times}$ is surjective.
\end{lemma}

\begin{proof}
Let $\im_1, \dots, \im_k$ be the maximal ideals of $R$ not containing $I$ and let $J = (\bigcap_i \im_i) \cap I$. 
By the Chinese Remainder Theorem, the map $$\rho: r + J \mapsto (r + I, r + \im_1, \dots, r + \im_k)$$ 
is a ring isomorphism from $R/J$ onto 
$R/I \times R/\im_1 \times \cdots \times R/\im_k$. 
Given $u \in (R/I)^{\times}$ we can find $v = \tilde{u} + J  \in (R/J)^{\times}$ such that
$\rho(v) = (u,  1 + \im_1, \dots, 1 + \im_k)$. 
Hence we have $u = \tilde{u} + I$. As $J \subset \Jac(R)$, we also have $\tilde{u} \in R^{\times}$.
\end{proof}

Given an ideal $I$ of $R$, we denote by $\pi_I$ the natural group epimorphism $R \rtimes_{\alpha} C \twoheadrightarrow (R/I) \rtimes_{\alpha + I} C$. Let us state the main result of this section.

\begin{theorem} \label{ThNielsenRankTwo}
Let $\gb, \gb' \in \V_2(G)$, $\eb \Doteq \pi_{ab}(b, a)$ and $R_C \Doteq R/(1 - \alpha)R$. 
\begin{enumerate}
\item 
The following are equivalent:
\begin{itemize}
\item[$(i)$]
The pairs $\gb$ and $\gb'$ are Nielsen equivalent.
\item[$(ii)$]
The pairs $\pi_I(\gb)$ and $\pi_I(\gb')$ are Nielsen equivalent for every\\
$I \in \left\{ (1 - \alpha)R, \nu(G) R\right\}$.
\item[$(iii)$] $\det_{\eb} \circ \pi_{ab}(\gb) = \pm \det_{\eb} \circ \pi_{ab}(\gb')$ and $\Delta_a(\gb) = \Delta_a(\gb')$, with $\det_{\eb}$ as in Remark \ref{DetE}.
\end{itemize}
\item
If $C$ is infinite or $R_C$ is finite then $\Delta_a$ is surjective and the above conditions are equivalent to
$
\Delta_a(\gb) = \Delta_a(\gb').
$
In this case $\Delta_a$ is a complete invariant of Nielsen equivalence for generating pairs.
\item
If $C$ is finite and $\Gab$ is infinite, then $\nielsen_2(G) = \max(\varphi(\vert C \vert)/2,1)\vert \Lambda^{\times}/T_{\Lambda} \vert$.
\end{enumerate}
\end{theorem}

\begin{proof}
$1.(i) \Rightarrow (ii)$. This follows from the $\Aut(F_2)$-equivariance of $\pi_I$.\\
$1.(ii) \Rightarrow (iii)$. We deduce the identity $\det_{\eb} \circ \pi_{ab}(\gb) = \pm \det_{\eb} \circ\pi_{ab}(\gb')$
from Theorem \ref{ThNielsenAbel} and the identity 
$\Delta_a(\gb) = \Delta_a(\gb')$ from Lemma \ref{LemDeltaAInvariant}.

$1.(iii) \Rightarrow (i)$. 
Suppose first that $C$ is infinite or $R_C$ is finite. 
By Theorem \ref{ThMTimesC}, we know that 
$\gb$ and $\gb'$ can be Nielsen reduced to $(r, a)$ and $(r', a)$ 
for some $r, r' \in R^{\times}$. 
By Lemma \ref{LemReductionModuloNuRank2} the pairs $\gb$ and $\gb'$ are Nielsen equivalent.
Suppose now that $C$ is finite and $R_C$ is infinite. 
By Proposition \ref{PropReductionToRowFree}, we know that 
$\gb$ and $\gb'$ can be Nielsen reduced to $(r, a^k)$ and $(r', a^{k'})$ 
for some $r, r' \in R^{\times}$ and $k, k' \in \Z$ such that 
$k \equiv \pm \det_{\eb} \circ \pi_{ab}(\gb)$ and 
$k' \equiv \pm \det_{\eb} \circ \pi_{ab}(\gb') $ modulo $\cC$. 
We deduce from Theorem \ref{ThNielsenAbel} that $k' \equiv \pm k \mod \cC$. 
Replacing $\gb$ by $\gb I_2$ if needed, we can assume that $a^k = a^{k'}$. 
Thanks to Remark \ref{RemarkDaUnicity}, we can argue as in the first part of the proof 
where $a$ is replaced by $a^k$, which proves 
the Nielsen equivalence of $\gb$ and $\gb'$.

$2$. We already showed in the proof of $(1)$ that $\Delta_a$ is injective if $C$ is infinite or $R_C$ is finite.
Thus we are left with the proof of $\Delta_a$'s surjectivity. Clearly, it suffices to show that
the natural map $R^{\times} \rightarrow (R/\nu(G) R)^{\times}$ is surjective.
This is trivial if $C$ is infinite since $\nu(G) = 0$ in this case.
So let us assume that $R_C$ is finite. Since we have $(1 - \alpha)\nu(G) = 0$, 
the ring element $\nu(G)$ belongs to every maximal ideal of $R$ which doesn't contain $1 - \alpha$. 
Hence it belongs to all but finitely many maximal ideals of $R$. Now Lemma \ref{LemSurjectivity} yields the conclusion.

$3$. This is a direct consequence of the characterization $1.iii$.
\end{proof}

We end this section with an algorithmic characterization of generating pairs.
\begin{proposition} \label{PropGeneratorsDelta}
 Assume $\nu(G)$ is nilpotent and let $\gb \in G^2$. Then the following are equivalent:
\begin{itemize}
\item[$(i)$] $\gb$ generates $G$.
\item[$(ii)$]$\sigma(\gb)$ generates $C$ and $D_a(\gb) \in \left(R/\nu(G) R\right)^{\times}$.
\end{itemize}
\end{proposition}

\begin{proof}
$(i) \Rightarrow (ii)$. This follows from $\sigma$'s surjectivity and Lemma \ref{LemmaDelta}.

$(ii) \Rightarrow (i)$. Since $\sigma(\gb)$ generates $C$ there is $\psi \in \Aut(F_2)$ such that 
$\sigma(\gb) \psi = (1_C, a)$. Replacing $\gb$ by $\gb\psi$, we can then assume that $\gb = (r, r' a)$ for some 
$r, r' \in R$. Since $\Delta_a(\gb) = r + \nu(G) R \in (R/\nu(G) R)^{\times}$ and $\nu(G) \in \Jac(R)$ we deduce that $r \in R^{\times}$. 
Therefore $\gb$ generates $G$.  
\end{proof}

\subsection{Nielsen equivalence of generating triples and quadruples} \label{SecN2N3}
In this section, we prove the last two assertions of Theorem \ref{ThN2}.
Since $\rk(G) = 2$ and $\dimK(R) \le 2$, Corollary \ref{CorStableRank} ensures that $G$ 
has only one Nielsen class of generating $n$-vectors for $n > 4$. 
Using Theorem \ref{ThMTimesC} in combination with 
a theorem of Suslin \cite[Theorem 7.2]{Sus77}, we show in Theorem \ref{ThTriplesAndQuadruples} below that this remains true if $n > 3$.
Recall that the map
$$\Phi_a: \rb \Gamma_{n -1}(R) \mapsto (\rb, a) \Aut(F_n)$$
defined in Theorem \ref{ThMTimesC} is a bijection from
$\U_{n - 1}(R)/\Gamma_{n - 1}(R)$ onto \\$\V_n(G)/\Aut(F_n)$ provided that $C$ is infinite.

\begin{lemma} \label{LemGEImpliesOneNielsenClass}
Let $n \ge 3$. If $R$ is a $\GE_{n - 1}$-ring, 
then $\nielsen_n(G) = 1$.  
\end{lemma}

\begin{proof}
Let $\gb \in \V_n(G)$. We shall show that
$\gb$ is Nielsen equivalent to $\gb_1 \Doteq (\rb_1, a)$ with $\rb_1 = (1, 0, \dots, 0) \in \U_{n -1}(R)$. 
As $n > \rk(G)$, the property $\Ni_n(a)$ holds by Theorem \ref{ThMTimesC}. 
Therefore $\gb$ can be Nielsen reduced to a vector of
the form $(\rb, a)$ with $\rb \in \U_{n - 1}(R)$. 
Since $R$ is a $\GE_{n - 1}$-ring and $R$ is completable by Lemma \ref{LemCompletable}, the group 
$\E_{n - 1}(R)$ acts transitively on $\U_{n - 1}(R)$. 
Hence $\rb$ can be transitioned to $\rb_1$ under the action of $\E_{n - 1}(R)$. 
Lemma \ref{LemGammaImpliesNielsen} implies that $\gb$ is Nielsen equivalent to $\gb_1$.
\end{proof}

Our forthcoming result on generating triples involves the following definitions from algebraic $K$-theory.
For every $n \ge 1$, the map $A \mapsto \begin{pmatrix} A & 0 \\0 & 1\end{pmatrix}$ defines an embedding from $\SL_n(R)$ into $\SL_{n + 1}(R)$, respectively from $\E_n(R)$ into $\E_{n + 1}(R)$.  Denote by $\SL(R)$ and $\E(R)$ the respective ascending unions. Then $\E(R)$ is normal in $\SL(R)$ and the group $\SK_1(R)$, the \emph{special Whitehead group of $R$}, is the quotient $\SL(R) /\E(R)$ (see, e.g., \cite{Mag02}). The next lemma shows in particular that the image in $\SK_1(R)$ of a matrix in $\SL_2(R)$ depends only on its first row.

\begin{lemma} \label{LemRho}
Let $R$ be any commutative ring with identity. Denote by $\hat{\E}_2(R)$ the normal closure of $\E_2(R)$ in $\SL_2(R)$. 
Let $\rho: \SL_2(R) \rightarrow \U_2(R)$ be defined by $\begin{pmatrix} a & b \\ c & d \end{pmatrix} \mapsto (a, b)$.
Then the map $\rho$ induces a bijection from
 $\SL_2(R)/\hat{\E}_2(R)$ onto $\U_2(R)/\hat{\E}_2(R)$.
\end{lemma}

\begin{proof}
For every $A, B \in \SL_2(R)$ the identity $\rho(AB) = \rho(A)B$ holds. Therefore the map 
$\hat{\rho}: A \Ehat_2(R) \mapsto \rho(A) \Ehat_2(R)$ is well defined. Let $(a, b) \in \U_2(R)$ and let $a', b' \in R$ be such that
$a a' + b b' = 1$. 
Then $A \Doteq \begin{pmatrix} a & b \\ -b' & a' \end{pmatrix} \in \SL_2(R)$ and $(a, b) = \rho(A)$, so that 
$\rho$, and hence $\hat{\rho}$ is surjective. Let us prove that $\hat{\rho}$ is injective. Consider for this $A, B \in \SL_2(R)$ such that 
$\hat{\rho}(A) = \hat{\rho}(B)$. Multiplying $A$ on the right by a matrix in $\Ehat_2(R)$ if needed, we can assume that 
$\rho(A) = \rho(B)$. Thus $\rho(AB^{-1}) = \rho(A)B^{-1} = \rho(B)B^{-1} = (1, 0)$, which shows that $AB^{-1} \in \E_2(R)$. The result follows.
\end{proof}

The \emph{Mennicke symbol} $\br{\rb}$ of $\rb \in \U_2(R)$ is the image in $\SK_1(R)$ of any matrix of $\SL_2(R)$ whose first row is $\rb$.
We are now in position to prove

\begin{theorem} \label{ThTriplesAndQuadruples}
The following hold:
\begin{itemize}
\item[$(i)$] If $C$ is infinite then $\V_3(G)/\Aut(F_3)$ surjects onto $\SK_1(R)$.
\item[$(ii)$] If $R$ is a $\GE_2$-ring, e.g., $C$ is finite, then $\nielsen_3(G) = 1$.
\item[$(iii)$] $\nielsen_4(G) = 1$. 
\end{itemize}
\end{theorem}

\begin{proof}
$(i)$. By Theorem \ref{ThMTimesC}, we can identify the two orbit sets $\V_3(G)/\Aut(F_3)$ and $\U_2(R)/\Gamma_2(R)$.
The classical properties of the Mennicke symbol \cite[Proposition VI.3.4]{Lam06} imply 
that the map $\br{\cdot}: \U_2(R) \rightarrow \SK_1(R)$ is $\Gamma_2(R)$-invariant. 
This yields a map $\U_2(R)/\Gamma_2(R) \rightarrow \SK_1(R)$. 
By Remark \ref{RemStableRankReduction}, the latter map is surjective. 

$(ii)$. This is Lemma \ref{LemGEImpliesOneNielsenClass} for $n = 3$.

$(iii)$. We can assume that $C$ is infinite since Lemma \ref{LemGEImpliesOneNielsenClass} applies otherwise. 
If $\dimK(R) \le 1$, then $\nielsen_4(G) = 1$ by 
Corollary \ref{CorStableRank}. Thus, we can also suppose that $R = \ZX$. 
Since $\V_4(G)/\Aut(F_4)$ identifies with $\U_3(R)/\Gamma_3(R)$ by Theorem \ref{ThMTimesC}
and since $\E_3(R)$ acts transitively on $\U_3(R)$ by \cite[Theorem 7.2]{Sus77}, we deduce that $\nielsen_4(G) = 1$.
\end{proof}

\subsection{T-systems of generating pairs and triples} \label{SecT2}
This section is dedicated to the proofs of Theorems \ref{ThT2} and \ref{ThT2GN}.
Recall that $\nielsen_n(G)$ denotes the number of Nielsen equivalence classes of generating $n$-vectors of $G$ and $\tsys_n$ denotes the number of $\T_n$-systems of $G$, both numbers may be infinite.
We refer the reader to Lemmas \ref{LemDerivation} and \ref{LemAutG} for the definition of the automorphisms $X_d$ and $Y_{\tau, \theta}$ used below.

\begin{sproof}{Proof of Theorem \ref{ThT2}}
$(i)$. This is a specialization of Theorems \ref{ThOneTSystem} and  \ref{ThTSystems} to $G = R \rtimes_{\alpha} C$. 
$(ii)$.
Consider the action of $\Aut(G)$ on $\V_3(G)/\Aut(F_3)$ defined by 
$$\phi \cdot (\gb \Aut(F_3)) = (\phi \gb) \Aut(F_3), \quad \gb \in \V_3(G), \phi \in \Aut(G).$$
Regarding the first inequality, it suffices to show that the stabilizer $S_{\gb}$ of $\gb \Aut(F_3)$ has index at most 
$\cAC \nielsen_2(G)$ in $\Aut(G)$ for every $\gb \in \V_3(G)$.
By Theorem \ref{ThMTimesC}.$i$, such a triple $\gb$ is Nielsen equivalent to $(r, s, a)$ for some $r, s \in R$. 
Since $R = rR + sR = \Z\br{\alpha^{\pm1}}$, we easily see that every automorphism $X_d$ stabilizes $\gb \Aut(F_3)$. 
For an automorphism $Y_{\tau, 1}$ of $G$, we observe that $\tau \in \Aut_{\Z}(R)$ is actually an $R$-automorphism, 
so that $\tau$ is the multiplication by some unit $u_{\tau}$ of $ R$. 
If $u_{\tau}$ is a trivial unit, we see that $Y_{\tau} = Y_{\tau, 1}$ stabilizes $\gb \Aut(F_3)$ considering conjugates of the first two components of $\gb$. 
If $A(C)$ contains an automorphisms $\theta$ which maps $a$ to $a^{-1}$, we let $\phi_{-1}$ be an automorphism of the form $Y_{\tau, \theta}$ whose image is $\theta$ through the natural map $\Aut(G) \twoheadrightarrow A(C)$. Otherwise we set $\phi_{-1} = 1$. Let $V$ be a transversal of $R^{\times}/T$. It follows from Lemma \ref{LemAutG} that $\{Y_{\tau} \phi_{-1}^{\epsilon};\, \epsilon \in \{0, 1\}, \tau(b) \in V \}$ is a transversal of $\Aut(G)/S_{\gb}$. Since $\nielsen_2(G) = \vert R^{\times}/T \vert$ by Theorem \ref{ThN2}.$ii$, we deduce that 
$\left\vert \Aut(G)/S_{\gb} \right\vert \le \cAC \nielsen_2(G)$, which completes the proof of the first inequality.

In order to prove the second inequality, we consider the action of $\Aut(G)$ on $\SK_1(R)$ defined by
$\phi \cdot \br{\rb} = \br{\phi(\rb)}$ for $(\phi, \rb) \in \Aut(G) \times \U_2(R)$ and where 
$\br{\rb}$ denotes the Mennicke symbol of $\rb$. The fact that this action is well-defined follows from Lemma \ref{LemAutG} and the classical properties of Mennicke symbols. Indeed, every automorphism 
$\phi \in \{ X_d, Y_t \, \vert \, d \in \Der(C, R), \tau \in \Aut_R(R)\}$ fixes every symbol $\br{\rb}$. Besides, the group automorphism $\phi_{-1}$ induces a ring automorphism of $R$, so that $\br{\rb} = \br{\rb'}$ implies $\br{\phi_{-1}(\rb)} = \br{\phi_{-1}(\rb')}$ for any two rows defining the same symbol. We actually showed that the $\Aut(G)$-action on $\SK_1(R)$ factors through an $A(C)$-action.
The map 
$(\rb, a) \Aut(F_3) \mapsto \br{\rb}$
induces an $\Aut(G)$-equivariant map $\mu$ from $\V_3(G)/\Aut(F_3)$ onto $\SK_1(R)$. 
As $\Aut(G)\backslash\SK_1(R) \simeq A(C)\backslash\SK_1(R)$, $\mu$ induces a surjective map from \\
$\Aut(G) \backslash\V_3(G)/\Aut(F_3)$ onto $ A(C)\backslash\SK_1(R)$, which yields the result.
\end{sproof}

With Theorem \ref{ThT2}, we observed that $\tsys_2(G) = 1$ holds if $C$ is infinite or $\Gab$ is finite.
With Corollary \ref{CorTTwoSystemsExample} below, we prove the first part of Theorem \ref{ThT2GN}, that is, 
$G$ can have arbitrarily many $\T_2$-systems when
 $C$ is finite but $\Gab$ isn't.

\begin{corollary} \label{CorTTwoSystemsExample}
Let $q = p^d$ and $N = q - 1$, with $p$ a prime integer and $d \ge 2$ an even integer. Let $\Phi_{N, p}(X)$ be the $N$-th cyclotomic polynomial over $\F_p$ and let $P \in \Z\br{X}$ be a monic polynomial of degree $d$ 
whose reduction modulo $p$ is an irreducible factor of $\Phi_{N, p}(X)$. 
Let $R = \Z\br{X}/(X - 1)I$ where $I =(p, P(X))$ is the ideal generated by $p$ and $P(X)$.
Then the image $\alpha$ of $X$ in $R$ is invertible. It generates a subgroup $C \subset R^{\times}$ with $N$ elements and
the number of 
$\T_2$-systems of 
$G = R \rtimes_{\alpha} C$ is $\tsys_2(G) = \varphi(N)/d$.
\end{corollary}

We will use the following straightforward consequence of Lemma \ref{LemAutG}:
\begin{lemma} \label{LemAutR}
Let $k \in \Z$. The following are equivalent:
\begin{itemize}
\item[$(i)$] There is $\theta \in A(C)$ such that $\theta(a) = a^k$.
\item[$(ii)$] The map $\alpha \mapsto \alpha^k$ induces a ring automorphism of $R$.
\end{itemize}
\end{lemma}
$\square$
\begin{lemma} \label{LemRCharacteristic}
The two following hold:
\begin{itemize}
\item[$(i)$]Let $g = r a^k \in G$ with $r \in R$, $k \in \Z$. Then $g$ centralizes its conjugacy class if and only if 
$
(1 - \alpha^k)^2 = 0 = (1 - \alpha)(1 - \alpha^k)r
$.
\item[$(ii)$] Let $\omega$ be the order of $\alpha$ in $R^{\times}$. Assume that $\omega = \cC$ and that 
for every $k \in \Z$, we have $(1 - \alpha^k)^2 \neq 0$ whenever $\alpha^k \neq 1$. 
Then $R$ is a characteristic subgroup of $G$.
\end{itemize}

\begin{proof}
Assertion $(i)$ is a direct consequence of the identity
$$
\br{g, hgh^{-1}} = (1 - \alpha^k) \left( (1 - \alpha^k)r - (1 - \alpha^{k'})r'\right), \text{ where } h = r'a^{k'}.
$$
In order to prove $(ii)$, consider $\phi \in \Aut(G)$ and write $\phi(b) = ra^k$ where $b$ is the identity of the ring $R$.
Since $b$ centralizes its conjugacy class, so does $\phi(b)$. By $(i)$, we have $(1 - \alpha^k)^2 = 0$, which yields $\alpha^k = 1$.
As $\omega = \cC$, we deduce that $\phi(b) = r$ and hence $\phi(R) = R$.
\end{proof}
\end{lemma}

\begin{sproof}{Proof of Corollary \ref{CorTTwoSystemsExample}}
The existence of the polynomial $P(X)$  is guaranteed by \cite[Theorem 2.47]{LN97}.
Let $\oP(X) \in \F_p\br{X}$ be the reduction of $P(X)$ modulo $p$.
By the Chinese Remainder Theorem, the ring $R$ identifies with $\Z \times \F_q$ where $\F_q = \Z_p\br{X} / (\oP(X))$ is the field with $q$ elements.
As a result, the element $\alpha$ identifies with $(1, x)$ where $x \in \F_q^{\times}$ is an $N$-th primitive root of unity. Thus
$C \simeq \F_q^{\times}$ and the ring automorphisms of $R$ induced by maps of the form 
$\alpha \mapsto \alpha^k$ correspond bijectively to powers of the Frobenius endomorphism of $\F_q$. 
Lemma \ref{LemRCharacteristic}'s hypotheses are easily checked so that $R$ is a characteristic subgroup of $G$. 
By Lemma \ref{LemAutR}, we have then $\vert A(C) \vert = d$ 
and hence $\vert A'(C) \vert = d$ for $d$ is even. By Theorem \ref{ThT2}, we obtain 
$\tsys_2(G) = \vert \Aut(C) / A'(C) \vert = \varphi(N)/d$.
\end{sproof}

\section{Baumslag-Solitar groups, split metacyclic groups and lamplighter groups} \label{SecLamplighterGroups}
This section is dedicated to the proofs of the Corollaries \ref{CorV2}, \ref{CorBS}, \ref{CorMetacyclic}, \ref{CorWreath} and \ref{CorZwrZ}.

The following lemma is a key ingredient in the proof of Corollary \ref{CorV2}.
\begin{lemma} \label{LemTSysAbel}
Let $G = R \rtimes_{\alpha} C$ be as in Section \ref{SecRC}. Assume that the natural map $R^{\times} \rightarrow R_C^{\times}$ is surjective and that $C$ is infinite or $R_C \Doteq R/(1 - \alpha)R$ is finite. Then $\Aut(G) \times \Aut(F_2)$ acts transitively on $\V_2(\Gab)$ where the action of $\Aut(G)$ is the action induced by the natural homomorphism $\Aut(G) \rightarrow \Aut(\Gab)$.
\end{lemma}

Recall that for $t \in R^{\times} \simeq \Aut_R(R)$,  the automorphism $Y_t$ of $G \simeq R \rtimes_{\alpha} C$ is defined by 
$rc \mapsto t(r)c$.

\begin{proof}
Let us assume first that $C$ is infinite. By Theorem \ref{ThNielsenAbel}, every generating pair of $\Gab = R_C \times C$ is Nielsen equivalent to $(u, a)$ for some $u \in R_C^{\times}$. Let $t$ be the multiplication by $u^{-1}$ on $R_C$. By hypothesis, we can find a lift $\tilde{u}$ of $u$ in $R$ which is moreover a unit. Let $\tau$ be the multiplication by $\tilde{u}^{-1}$ on $R$. Then the automorphism $Y_t \in \Aut(\Gab)$ is induced by 
$Y_{\tau} \in \Aut(G)$ and we have $Y_t(u, a) = (\ob, a)$ where $\ob$ denotes the identity of $R_C$. Therefore every generating pair of $G$ is in the orbit of $(\ob, a)$ under the action of $\Aut(G) \times \Aut(F_2)$. 

Assume now that both $C$ and $R_C$ are finite and let $d$ be the greatest common divisor of $\cC$ and $\vert R_C \vert$.
For $\gb \in \Gab$, we define $\det(\gb)$ as in Corollary \ref{CorDet}, considering the decomposition $R_C \times C$. The latter corollary implies that generating pairs with the same determinant are Nielsen equivalent. Hence it suffices to prove that the orbit of an arbitrary generating pair $(x, y) \in \V_2(\Gab)$ contains a pair of determinant $1 \in \Z_d$. By Lemma \ref{LemSurjectivity}, there is a lift $u$ of $\det(\gb)$ in $R_C^{\times}$ and by hypothesis, there is in turn a lift $\tilde{u}$ of $u$ in $R^{\times}$. Reusing the notation of the previous paragraph, we see that $Y_t(x, y)$ is a generating pair of determinant $1$. The proof is then complete.
\end{proof}

\begin{sproof}{Proof of Corollary \ref{CorV2}}
Since $\tsys_2(G) = 1$ by Theorem \ref{ThT2}.$i$, the group $\Aut(G)$ acts transitively on $\V_2(G) / \Aut(F_2)$. 
Therefore the Nielsen equivalence classes of generating pairs have the same number of elements. Let us establish the formula.
By Lemma \ref{LemSurjectivity}, the natural map $R^{\times}  \rightarrow (R_C)^{\times}$ is an epimorphism.
Since $\Aut(G) \times \Aut(F_2)$ acts transitively on $\V_2(\Gab)$ by Lemma \ref{LemTSysAbel}, the number of preimages of $\ogb$ in 
$\V_2(G)$ with respect to the abelianization homomorphism $\pi_{ab}$ does not depend on $\ogb$. Hence it suffices to compute this number for $\ogb = (\overline{b}, a)$ where $\ob$ denotes the image of $b$ in $\Gab$. 
A generating pair 
$\gb \in \V_2(G)$ which maps to $(\ob, a)$ via $\pi_{ab}$ is of the form $(r, sa)$ with $r \in 1 + (1 - \alpha)R$ and 
$s \in (1 - \alpha)R$. It follows from Lemma \ref{LemGeneratorsOfM}.$i$ that a pair of this form generates $G$ if and only if $r$ is, in addition, a unit. Therefore the number of preimages of $\ogb$ is $\frac{\vert R^{\times} \vert}{ \vert (R_C)^{\times} \vert} \frac{\vert R \vert}{\vert R_C\vert}$.
\end{sproof}

\begin{sproof}{Proof of Corollary \ref{CorMetacyclic}}
$(i)$. Since $G = \Z \times \Z_l$ by hypothesis, the result follows from Theorem \ref{ThNielsenAbel}.

$(ii)$.
By Theorem \ref{ThN2}.$ii$, we have 
$$\nielsen_2(G) = \left\vert (R/\nu(G) R)^{\times} /\langle \pm \alpha \nu(G)R \rangle\right\vert$$
with $R = \Z_k$. Thus $\nielsen_2(G) = \frac{\varphi(\lambda)}{\omega}$ follows from the definitions of $\lambda$ and $\omega$.

$(iii)$. By Theorem \ref{ThT2}.$i$ we have $\tsys_2(G) = 1$. Since $\Z_k$ is a $\GE$-ring, it follows from Theorem \ref{ThN2}.$iv$ that $\nielsen_3(G) = 1$.

$(iv)$. Corollary \ref{CorV2} applies with $R = \Z_k$ and $R_C = \Z_k/(1 - \alpha)\Z_k \simeq \Z_e$.
\end{sproof}

\begin{sproof}{Proof of Corollary \ref{CorBS}}
$(i)$. 
Let $G = BS(1, l)$.
By Theorem \ref{ThN2}.$ii$, we have 
$\nielsen_2(G) = \left\vert R^{\times} /\langle \pm \alpha \rangle\right\vert$ 
with $R = \Zl$ 
and $\alpha = l$. 
The prime divisors of $l$ form a basis of a free Abelian subgroup of $R^{\times}$ of index $2$. Thus $\nielsen_2(G)$ is finite if only if $l = \pm p^d$ for some prime $p$ and some $d \ge 0$. If $d = 0$, then $R = \Z$ and clearly $\nielsen_2(G) = 1$. Otherwise, 
$\nielsen_2(G) = \left\vert \langle \pm p \rangle  / \langle \pm p^d \rangle \right\vert = d$.

$(ii)$. By Theorem \ref{ThT2}.$i$ (equivalently Brunner's theorem \cite{Bru74}) we have $\tsys_2(G) = 1$. 
Since $\Zl$ is Euclidean, it follows from Theorem \ref{ThN2}.$iv$ that $\nielsen_3(G) = 1$.
\end{sproof}

We consider now the two-generated lamplighter groups, i.e., the restricted wreath products of the form
$G = \Z_k \wr \Z_l$ with $k, l \ge 0$ and $k, l \neq 1$. 
Such a group $G$ reads also as 
$G = R \rtimes_a C$ with
$C = \Z_k = \langle a \rangle$ and $R = \Z_k \br{C} \simeq \Z_k\br{X}/(X^l - 1)$.
As before, we denote by $T$ the subgroup of $R^{\times}$ generated by $-1$ and $a$.
We also set 
$
\Lambda \Doteq R/\nu(G) R
$
and
$
T_{\Lambda} \Doteq \pi_{\nu(G) R}(T)
$,  like in Section \ref{SecNielsenPairs}.
Corollary \ref{CorWreath} will be obtained in combining Corollaries \ref{CorWreathOneTTwoSystem} and \ref{CorWreathNTwo} below.
\begin{corollary} \label{CorWreathOneTTwoSystem}
Let $k, l \ge 0$ with $k, l \neq 1$ and let $G = \Zwr$. Then the following hold.
\begin{itemize}
\item[$(i)$]
$\tsys_2(G) = 1$.
\item[$(ii)$]
If $\Z_k$ is finite or $\Z_l$ is infinite, then 
$
\nielsen_2(G) = \left\vert \Lambda^{\times} /T_{\Lambda} \right\vert.
$
\item[$(iii)$]
If $\Z_k$ or $\Z_l$ is finite, then $\nielsen_3(G) = 1$.
\end{itemize}
\end{corollary}
\begin{proof}
$(i)$. If $\Z_k$ is finite, or $\Z_l$ is infinite, then $\tsys_2(G) = 1$ by Theorem \ref{ThT2}.$i$. 
Otherwise, Theorem \ref{ThT2}.$ii$ applies and $\tsys_2(G) \le \vert \Aut(C) / A'(C) \vert$.
It is easy to see that the map $a \mapsto a^i$ induces a ring automorphism of $R$ for every $i$ coprime with $l$. 
Thus $A'(C) = \Aut(C)$ by Lemma \ref{LemAutR}, which implies $\tsys_2(G) = 1$.

$(ii)$. This is an immediate consequence of Theorem \ref{ThN2}.$ii$. 
 
$(iii)$. If $\Z_k$ is finite then $R$ is a $\GE$-ring by Lemma \ref{LemArtinianCoefficientAndGE}. 
If $\Z_l$ is finite then $R$ is $\GE$-ring by Theorem \ref{ThZCGE}.
Therefore $\nielsen_3(G) = 1$ by Theorem \ref{ThN2}.$iv$.
\end{proof}
\begin{corollary} \label{CorFiniteWreathProduct}
Assume that both $\Z_k$ and $\Z_l$ are finite and non-trivial. 
Given a prime divisor $p$ of $k$, we denote by $\nu_l(p, d)$ the number of distinct 
irreducible factors of $$1 + X + \cdots + X^{l - 1}$$ in $\Z_p \br{X}$ which are monic of degree $d$. 
Let $l' = 2l$ if $k \neq 2$, $l' = l$ otherwise.
Then we have 
$$\nielsen_2(\Zwr) = \frac{k^{l - 1}}{l'} \prod_{p, d} (1 - \frac{1}{p^d})^{\nu_l(p, d)}$$ 
where $p$ ranges over the prime divisors of $k$ and $d$ over the positive integers.
\end{corollary}

The following lemma makes easy the task of computing the cardinality of the unit group in each finite ring under consideration.

\begin{lemma} \cite[Exercise 44]{Ste12} \label{LemUnitsOfAFiniteRing}
Let $R$ be a finite ring. Then $$\vert R^{\times} \vert = \vert R \vert \prod_{\im}(1 - \frac{1}{\vert R / \im \vert})$$ where $\im$ ranges over the maximal ideals of $R$. 
\end{lemma}

\begin{sproof}{Proof of Corollary \ref{CorFiniteWreathProduct}}
Since $\nu(G) R = \Z_k \nu(G)$, the ring $\Lambda$ has $k^{l - 1}$ elements.
Each maximal ideal $\im$ is generated by a prime divisor $p$ of $k$ 
and the image in $\Lambda$ of a polynomial $P \in \Z_k\br{X}$ 
whose reduction modulo $p$ is an irreducible monic factor of $1 + X + \cdots + X^{l - 1}$.
Hence $\Lambda /\im = \F_{p^d}$ where $d$ is the degree of $P$. 
Thus 
$\vert \Lambda^{\times} \vert = 
k^{l - 1}\prod_{p, d} (1 - \frac{1}{p^d})^{\nu_l(p, d)}$ by Lemma \ref{LemUnitsOfAFiniteRing} 
and we conclude the proof in observing that $l' = \vert T_{\Lambda} \vert$.
\end{sproof}
Given a prime divisor $p$ of $k$, we denote by $\mu_l(p, d)$ the number of distinct 
irreducible factors of $1 - X^l$ in $\Z_p \br{X}$ which are monic of degree $d$. 
Using Lemma \ref{LemUnitsOfAFiniteRing}, it is straightforward to establish the formula
$$
\left\vert (\Z_k\br{\Z_l})^{\times}\right\vert = k^l\prod_{p, d} (1- \frac{1}{p^d})^{\mu_l(p, d)}.
$$
where $p$ ranges over the prime divisors of $k$.
\begin{corollary} \label{CorFiniteWreathProductV2}
Assume that both $\Z_k$ and $\Z_l$ are finite and non-trivial. 
Then we have 
$$\left\vert \V_2(\Zwr) \right\vert  = \frac{k^{l -1}}{\varphi(k)}\left\vert (\Zkl)^{\times}\right\vert  
\left\vert \V_2(\Z_k \times \Z_l) \right\vert$$
and the number of elements in a Nielsen equivalence class of generating pairs is 
$$
l' k^{l - 1}\left\vert \V_2(\Z_k \times \Z_l) \right\vert.
$$
where $l'$ is as in Corollary \ref{CorFiniteWreathProduct}.
\end{corollary}
\begin{proof}
Let $G = \Zwr$. As $\Zkl/(1 - \alpha) \simeq \Z_k$, we have $\Gab \simeq \Z_k \times \Z_l$. We obtain the first formula by applying Corollary \ref{CorV2} with $R = \Zkl$ and $R_C = \Z_k$.
By the same corollary and Corollary \ref{CorFiniteWreathProduct}, the Nielsen equivalence classes of generating pairs have the same number of elements, given by
$$\frac{\left\vert \V_2(G) \right\vert}{\nielsen_2(G)} = 
\frac{l'k^l}{\varphi(k)} \prod_{p, d} (1- \frac{1}{p^d})^{\mu_l(p, d) - \nu_l(p, d)}
\left\vert \V_2(\Gab) \right\vert$$
where $p$ ranges over the prime divisors of $k$. 
The integer $\mu_l(p, d) - \nu_l(p, d)$ is the number of monic irreducible polynomials in 
$\Z_p \br{X}$ of degree $d$ which divides $1 - X^l$ but not $1 + X + \cdots + X^{l -1}$. 
Therefore $\mu_l(p, d) - \nu_l(p, d) = 1$ if $d = 1$ and it cancels otherwise. 
Thus we have 
$\prod_{p, d} (1- \frac{1}{p^d})^{\mu_l(p, d) - \nu_l(p, d)} = \prod_p (1- \frac{1}{p}) = \frac{\varphi(k)}{k}$, 
which gives the result
\end{proof}

\begin{corollary} \label{CorWreathNTwo}
Let $k, l \ge 0$ and $k, l \neq 1$.
\begin{itemize}
\item[$(i)$]
Assume that $\Z_k$ is finite and $\Z_l$ is infinite. Then 
$\nielsen_2(G)$ is finite if and only if $k$ is prime; in this case $\nielsen_2(G) = \max(\frac{k - 1}{2}, 1)$.
\item[$(ii)$]
Assume that $\Z_k$ is infinite and $\Z_l$ is finite. Then
$\nielsen_2(G)$ is finite if and only if $l \in \{2, 3, 4, 6\}$; in this case $\nielsen_2(G) = 1$.
\end{itemize}
\end{corollary}

\begin{proof}
$(i)$. 
The result follows from Corollary \ref{CorWreathOneTTwoSystem}.$ii$ and the isomorphisms
$$ 
\begin{array}{lllllllll}
(\ZkX)^{\times} &\simeq& \Z_k^{\times} &\times &  U_X  &\times & U_{X^{-1}} &\times& \Z^{\rho}.\\
T                        &\simeq& \{\pm 1 \}              &\times & \{1\}  &\times & \{1\}         &\times & \Z.
\end{array}
$$
where $\rho$ is the number of prime divisors of $k$ and $U_Y = 1 + Y\nil(\Z_k)\br{Y}$ with $Y \in \{X^{\pm 1}\}$ 
(see e.g., \cite[Exercise 3.17]{Wei13} where the units in the ring of Laurent polynomials are determined).

$(ii)$.
As $\Lambda = \Z\br{X} / (1 + X + \cdots + X^{l - 1})$, the ring $\Lambda$ identifies with $\mODo$ as defined in Lemma \ref{LemUnitsOfOD} and where $\mD$ is the set of divisors of $l$.
By Lemma \ref{LemUnitsOfOD}, the group $\Lambda^{\times}$ is finite if and only if $l \in \{2, 3, 4, 6 \}$; in this case 
 the equality $\Lambda^{\times} = T_{\Lambda}$ holds. 
 Since $\nielsen_2(G) = \max(\varphi(l)/2, 1) \left\vert \Lambda^{\times} /T_{\Lambda} \right\vert$
 by Theorem \ref{ThN2}.$iii$, the result follows. 
\end{proof}

We conclude with the group $G = \Z \wr \Z $, which isomorphic to $\ZX \rtimes_X \Z$.
\begin{sproof}{Proof of Corollary \ref{CorZwrZ}}
It follows from Lemma \ref{LemRCharacteristic} that $R = \ZX$ is characteristic in $G$.
The inequality $\nielsen_3(G) \le 2 \tsys_3(G)$ is then a consequence of Theorem \ref{ThT2}.$ii$.
The implication $(i) \Rightarrow (ii)$ is obvious while the equivalence $(i) \Leftrightarrow (iii)$ results from Theorem \ref{ThMTimesC} and Lemma \ref{LemGECriteria}.$i$.
In order to prove $(ii) \Rightarrow (i)$, we assume that $(ii)$ holds true, fix $\gb_0 \in \V_2(G)$ and let $\gb$ be an arbitrary generating triples of $G$. As $\tsys_3(G) = 1$ by hypothesis, we deduce that $\gb$ is Nielsen equivalent to a triple of the form $(1_G, \gb_1)$ with $\gb_1 \in \V_2(G)$. By Corollary \ref{CorWreathOneTTwoSystem}.$ii$, we have $\nielsen_2(G) = 1$, so that $(1_G, \gb_1)$ is Nielsen equivalent $(1_G, \gb_0)$. Therefore $\nielsen_3(G) = 1$. 
\end{sproof}

\bibliographystyle{alpha}
\bibliography{Biblio}
\end{document}